\documentclass[leqno]{amsart}
\usepackage{amsfonts,amssymb,amsmath,amsgen,amsthm}
\usepackage{hyperref,color}
\usepackage{pdfsync}
\usepackage{bbm}
\usepackage{bm}

\theoremstyle{plain}
\newtheorem{theorem}{Theorem}[section]
\newtheorem{definition}[theorem]{Definition}
\newtheorem{lemma}[theorem]{Lemma}

\newtheorem{proposition}[theorem]{Proposition}

\theoremstyle{remark}

\def\R{{\mathbb R}}
\def\T{{\mathbb T}}

\def\({\left(}
\def\){\right)}
\def\<{\left\langle}
\def\>{\right\rangle}

\def\ge{\geqslant}

\def\1{{\mathbf 1}}

\def\d{{\partial}}
\def\eps{\varepsilon}

\def\g{\gamma}

\def\e{\varepsilon}
\def\R{{\mathbf R}}
\def\T{{\mathbb T}^{d}}

\def\rrho{\sqrt{\rho}}
\def\({\left(}
\def\){\right)}
\def\<{\left\langle}
\def\>{\right\rangle}

\def\ge{\geqslant}

\def\1{{\mathbf 1}}

\newcommand{\dive}{\mathop{\mathrm {div}}}

\newcommand{\Sy}{\mathbb{S}}
\newcommand{\K}{\mathbb{K}}

\newcommand{\be}{\beta\left(\frac{|w_{\e}|^2}{2}\right)}
\newcommand{\bep}{\beta'\left(\frac{|w_{\e}|^2}{2}\right)}
\newcommand{\bedp}{\beta''\left(\frac{|w_{\e}|^2}{2}\right)}
\newcommand{\rn}{\rho_\e}
\newcommand{\rrn}{\sqrt{\rho_\e}}
\newcommand{\un}{u_{\e}}
\newcommand{\wn}{w_{\e}}
\newcommand{\pt}{\partial_t}
\newcommand{\re}{\rho_\e}
\newcommand{\rre}{\sqrt{\rho_\e}}
\newcommand{\ue}{u_{\e}}
\newcommand{\we}{w_{\e}}

\newcommand{\he}{h_{\e}}
\newcommand{\ve}{v_{\e}}
\newcommand{\gie}{g_{\e}}
\newcommand{\phie}{\phi_{\e}}
\newcommand{\rei}{q_{\e}}

\DeclareMathOperator{\diver}{div}
\DeclareMathOperator{\symmD}{D}

\numberwithin{equation}{section}

\date\today

\title[QNS equations]{Global Existence of Finite Energy Weak Solutions of Quantum Navier-Stokes Equations}

\author[P.~Antonelli]{Paolo Antonelli}
\address[P.~Antonelli]{Gran Sasso Science Institute \\ via Crispi 7 \\ 67100 L'Aquila (Italy).}
\email{paolo.antonelli@gssi.it}
\author[S.~Spirito]{Stefano Spirito}
\address[S.~Spirito]{Gran Sasso Science Institute \\ via Crispi 7 \\ 67100 L'Aquila (Italy).}
\email{stefano.spirito@gssi.infn.it}


\begin{document}
\begin{abstract}
In this paper we consider the Quantum Navier-Stokes system both in two and in three space dimensions and prove global existence of finite energy weak solutions for large initial data. In particular, the notion of weak solutions is the standard one. This means that the vacuum region are included in the weak formulations. In particular, no extra term like damping or cold pressure are added to the system in order to define the velocity field in the vacuum region.  The main contribution of this paper is the construction of a regular approximating system consistent with the effective velocity transformation needed to get necessary a priori estimates.  
\end{abstract}

\maketitle

\section{Introduction}
\label{intro}
In this paper we study the Quantum-Navier-Stokes (QNS) system on $(0, T)\times\Omega$,
\begin{equation}\label{eq:qns}
\begin{aligned}
&\pt\rho+\dive(\rho u)=0\\
&\pt(\rho u)+\dive(\rho u\otimes u)+\nabla\rho^{\gamma}-2\nu\dive(\rho Du)-2\kappa^2\rho\nabla\left(\frac{\Delta\rrho}{\rrho}\right)=0,
\end{aligned}
\end{equation}
with initial data
\begin{equation}\label{eq:id}
\begin{aligned}
\rho(0,x)&=\rho^0(x),\\
(\rho u)(0,x)&=\rho^0(x)u^0(x).
\end{aligned}
\end{equation} 
The domain $\Omega$ we consider is the $d$-dimensional torus with $d=2,3$. The unknowns $\rho, u$ represent the mass density and the velocity field of the fluid, respectively, $\nu$ and $\kappa$ are positive constants and they are called the viscosity and the dispersive coefficients.

The above system belongs to a wider class of fluid dynamical evolution equations, called Navier-Stokes-Korteweg systems, which read
\begin{equation}\label{eq:nsk}
\begin{aligned}
&\d_t\rho+\diver(\rho u)=0\\
&\d_t(\rho u)+\diver(\rho u\otimes u)+\nabla p=\diver\mathbb S+\diver\mathbb K,
\end{aligned}
\end{equation}
where $\mathbb S=\mathbb{S}(\nabla u)$ is the viscosity stress tensor
\begin{equation*}
\mathbb S=h(\rho)\symmD u+g(\rho)\diver u\mathbb I,
\end{equation*}
and $\mathbb K=\mathbb{K}(\rho,\nabla\rho)$ the capillarity (dispersive) term, defined through
\begin{equation*}
\mathbb K=\left(\rho\diver(k(\rho)\nabla\rho)-\frac12(\rho k'(\rho)-k(\rho))|\nabla\rho|^2\right)\mathbb I-k(\rho)\nabla\rho\otimes\nabla\rho.
\end{equation*}
The QNS system \eqref{eq:qns} is given by choosing in \eqref{eq:nsk} the capillarity coefficient to be $k(\rho)=\frac{\kappa^2}{\rho}$.

Furthermore, similar systems arise also in the description of quantum fluids. For example the inviscid system, i.e. \eqref{eq:qns} with $\nu=0$, is the well known Quantum Hydrodynamics (QHD) model for superfluids \cite{LL}. Global existence of finite energy weak solutions for the QHD system has been studied in \cite{AM} and \cite{AM2}. Inviscid systems with a general capillarity tensor are also studied extensively, for example in \cite{BGDD} the local well-posedness in high regularity spaces of the Euler-Korteweg system is treated. Recently in \cite{AH} the global well-posedness of the same system for small irrotational data was proved.
The viscous correction term in \eqref{eq:qns} has been also derived in \cite{BM}, by closing the moments for a Wigner equation with a BGK term. For more details about the derivation of the QNS system we refer the reader to \cite{J2}.

The main result we are going to prove in our paper is the existence of global in time finite energy weak solutions for the Cauchy problem \eqref{eq:qns}, \eqref{eq:id}. This is the first result of global existence for finite energy weak solutions to a Navier-Stokes-Korteweg system in several space dimensions. For the one dimensional case, in \cite{J3} the global existence of weak solutions for the QNS system \eqref{eq:qns} is proved. Furthermore, in \cite{GLF} the authors consider a large class of NSK systems in one dimension, for which they prove the existence of global in time finite energy weak solutions. We also mention \cite{CCDZ} where the authors show the existence of global classical solutions around constant states in one space dimension. 
Concerning the multidimensional setting, in \cite{H} the existence of global strong solutions to \eqref{eq:qns} is shown, by choosing a linear pressure and $\kappa=\nu$.

A global existence result for \eqref{eq:qns}, \eqref{eq:id} with finite energy initial data was already obtained by J\"ungel in \cite{J} in the case $\kappa>\nu$ and $\gamma>3$. However, the notion of weak solutions in \cite{J} requires test functions of the type $\rho\phi$, with $\phi$ smooth and compactly supported. This particular choice of such test functions does not consider the nodal region $\{\rho=0\}$ in the weak formulation, where there are the main difficulties in dealing with the convective term and it was introduced in \cite{BDL} to prove a global existence result for a Navier-Stokes-Korteweg system \eqref{eq:nsk} with a specific choice of viscosity and capillarity coefficients. A similar analysis is also done in \cite{Ji} for the case $\kappa<\nu$.

Furthermore, some global existence results by using the classical notion of weak solutions have been shown by augmenting the system \eqref{eq:qns} with some additional terms: for example, \cite{GLV} considers a cold pressure term, whereas in \cite{VY2} damping terms are added. Those augmented systems ensure that the velocity field is well defined also in the vacuum region and it lies in some suitable Lebesgue or Sobolev spaces. From such a priori estimates it is then possible to infer the sufficient compactness properties for the weak solutions, in particular to deal with the convective term in the vacuum region.

When $\kappa=0$ in \eqref{eq:qns}, global existence results for finite energy weak solutions have been recently obtained by \cite{VY1} and \cite{LX}. One of the main tools to treat the convective term is the Mellet-Vasseur inequality \cite{MV}. There the authors prove the compactness of finite energy weak solutions for the Navier-Stokes equations with degenerate viscosity by obtaining a logarithmic improvement to the usual energy estimates, namely they show the quantity
\begin{equation*}
\rho|u|^2\log\left(1+|u|^2\right)
\end{equation*}
is uniformly bounded in $L^\infty_tL^1_x$.

The presence of the dispersive term in \eqref{eq:qns}, however, prevents to directly prove a Mellet-Vasseur type inequality. This was indeed already remarked in \cite{VY1}, where the authors can only prove an approximate estimate by exploiting the extra damping terms and a truncation technique for the mass density.

In \cite{AS} we overcome this difficulty by using an alternative formulation for \eqref{eq:qns} in terms of an effective velocity $w=u+c\nabla\log\rho$. In this way it is possible to tune the viscosity and capillarity coefficients such that the dispersive term vanishes in the new formulation. The Mellet-Vasseur inequality is proved then for the auxiliary system and, by using the a priori bounds obtained from a BD \cite{BD,BDcras} type estimate, we prove the compactness of solutions to \eqref{eq:qns}, \eqref{eq:id}. We refer to \cite{AS} for a more detailed discussion on the stability properties of \eqref{eq:qns}, \eqref{eq:id}. 
We mention \cite{BGZ} and \cite{BDZ}, where a similar effective velocity was used to study fluid dynamical systems with a two-velocity formulation. We also refer to \cite{J3} for a further introduction on models where similar effective velocities are considered.

In this paper we continue our analysis of system \eqref{eq:qns}, \eqref{eq:id} by showing the global existence of finite energy weak solutions. The main difficulty here is to construct a sequence of approximating solutions which satisfy the a priori bounds in \cite{AS}. More precisely, we need to consider an approximating system with the following properties: first of all, it must retain all the a priori estimates, such as the energy and the BD entropy estimates. This further implies that the approximating system must be consistent with the transformation performed in \cite{AS} in terms of the effective velocity. Moreover, we need that the auxiliary system satisfies a Mellet-Vasseur type estimate. Finally, the approximating solutions must be regular.
We notice that standard approximation procedures based on Faedo-Galerkin method can not be used here since the a priori estimates in \cite{AS} heavily depend on the structure of the system.

The approximating system we are going to study is the following one
\begin{equation*}
\begin{aligned}
&\d_t\rho_\eps+\diver(\rho_\eps u_\eps)=0,\\
&\d_t(\rho_\eps u_\eps)+\diver(\rho_\eps u_\eps\otimes u_\eps)+\nabla(\rho_\eps^\gamma+p_\eps(\rho_\eps))+\tilde p_\eps(\rho_\eps)u_\eps=\diver \mathbb S_\eps+\diver \mathbb K_\eps,
\end{aligned}\end{equation*}
where $p_\eps(\rho_\eps)$ is a cold pressure term, $\tilde p_\eps(\rho_\eps)u_\eps$ is a damping term, $\mathbb S_\eps$ and $\mathbb K_\eps$ are the approximating viscosity and capillarity tensors, respectively. As we will see below, the cold pressure term will give us the higher integrability a priori bounds crucial to prove the global regularity of the approximating solutions. However, this introduces some difficulties in the analysis, first of all that prevents to obtain a Mellet-Vasseur type estimate. To overcome this problem we then add the damping term, with a suitable choice of the coefficient $\tilde p_\eps(\rho_\eps)$ such that in the auxiliary system written in terms of the effective velocity the cold pressure cancels. 
In order to show the convergence to zero of the cold pressure and damping terms we need additional a priori estimates. We manage to get further integrability properties for those singular terms by considering a regularized viscous stress tensor, similarly to \cite{LX}.
On the other hand, this requires that also the capillarity tensor has to be regularized accordingly; this is necessary so that the approximating system is consistent with the transformation through the effective velocity, as already remarked above. We will thus consider a regularization for the capillarity tensor such that it can be transformed as a part of the effective viscous tensor.
Moreover, this is the good approximation for the capillarity tensor since this yields the necessary a priori bounds on the mass density.

We conclude this introduction by a comparison with the result in \cite{AS}. The compactness holds for any $\nu, \kappa>0$ positive such that $\kappa<\nu$. In the two dimensional case we prove the existence result for the same range $\kappa<\nu$, while in the three dimensional case we consider $\nu$ and $\mu$ at the same scale, namely $\kappa<\nu<\alpha\kappa$ for some $\alpha>1$. However, it is worth to point out that no smallness assumption on $\nu$ and $\kappa$ are assumed. 
\\
Our paper is structured as follows: in Section \ref{sect:defs} we introduce the notations and definitions, in Section \ref{sec:app} we study the approximating system and we show some useful identities. Then, in Section \ref{sec:apriori} we prove the a priori estimate we need. Finally, in Section \ref{sec:proof} we prove the Theorem \ref{teo:main1} and \ref{teo:main2} and in Section \ref{sec:appex} we prove the global existence of smooth solutions for the approximating system. 

\section{Notations, Definitions and Main Result}\label{sect:defs}
In this section we are going to fix the notations used in the paper, to give the precise definition of weak solution for the system \eqref{eq:qns} and to state our main results.\\

\subsection{Notations}
Given $\Omega\subset\R^3$, the space of compactly supported smooth functions will be $\mathcal{D}((0,T)\times\Omega)$. We will denote with $L^{p}(\Omega)$ the standard Lebesgue spaces and with $\|\cdot\|_p$ their norm. The Sobolev space of $L^{p}$ functions with $k$ distributional derivatives in $L^{p}$ is $W^{k,p}(\Omega)$ and in the case $p=2$ we will write $H^{k}(\Omega)$. The spaces $W^{-k,p}(\Omega)$ and $H^{-k}(\Omega)$ denote the dual spaces of $W^{k,p'}(\Omega)$ and $H^{k}(\Omega)$ where  $p'$ is the H\"older conjugate of $p$. Given a Banach space $X$ we use the the classical Bochner space for time dependent functions with value in $X$, namely $L^{p}(0,T;X)$, $W^{k,p}(0,T;X)$ and $W^{-k,p}(0,T;X)$. Finally, $Du=(\nabla u+(\nabla u)^T)/2$ is the symmetric part of the gradient and $Au=(\nabla u-(\nabla u)^T)/2$ the antisymmetric part. In what follows, $C$ will be any constant depending on the data of the problem but independent on $\e$. Moreover, $\e$ will be always less than a small $\e_f$ depending only on $\g$, $\nu$ and $\kappa$, which will be chosen in the sequel.\\
\subsection{Weak Solutions}
We first recall two alternative ways to write the third order tensor term, which will be very useful in the sequel:
\begin{equation}\label{eq:quantum}
2\rho\nabla\left(\frac{\Delta\rrho}{\rrho}\right)=\dive(\rho\nabla^2\log\rho)=\nabla\Delta\rho-4\dive(\nabla\rrho\otimes\nabla\rrho).
\end{equation}
Then, by using \eqref{eq:quantum}, we can consider the following definition of weak solutions. 
\begin{definition}\label{def:ws}
A pair $(\rho, u)$ with $\rho\geq0$ is said to be a weak solution of the Cauchy problem \eqref{eq:qns}-\eqref{eq:id} if 
\begin{enumerate}
\item Integrability conditions:
\begin{align*}
&\rho\in L^{\infty}(0,T;L^{1}\cap L^{\gamma}(\T)),\\
&\rrho u\in L^{\infty}(0,T;L^{2}(\T)),\\
&\rrho\in L^{\infty}(0,T;H^{1}(\T)).\\
\end{align*}
\item Continuity equation:
\begin{equation*}
\int \rho^0\phi(0)+\iint\rho\phi_t+\rrho\rrho u\nabla\phi=0,
\end{equation*}
for any $\phi\in C_c^{\infty}([0,T);C^{\infty}(\T))$.\\
\item Momentum equation:
\begin{equation*}
\begin{aligned}
&\int \rho^0u^0\psi(0)+\iint\rrho(\rrho u)\psi_t+\rrho u\otimes\rrho u\nabla\psi+\rho^{\gamma}\dive\psi\\
&-2\nu\iint(\rrho u\otimes\nabla\rrho)\nabla\psi-2\nu\iint(\nabla\rrho\otimes\rrho u)\nabla\psi\\
&+\nu\iint\rrho\rrho u\Delta\psi+\nu\iint\rrho\rrho u\nabla\dive\psi\\
&-4\kappa^2\iint(\nabla\rrho\otimes\nabla\rrho)\nabla\psi+2\kappa^2\iint\rrho\nabla\rrho\nabla\dive\psi=0,\\
\end{aligned}
\end{equation*}

for any $\psi\in C_c^{\infty}([0,T);C^{\infty}(\T))$. 
\item Energy Inequality: if

\begin{equation*}
E(t)=\int\frac12\rho|u|^2+\frac{\rho^\gamma}{\gamma-1}+2\kappa^2|\nabla\sqrt{\rho}|^2,
\end{equation*}
then the following energy inequality is satisfied for a.e. $t\in(0,T)$ 
\begin{equation*}
E(t)\leq E(0).
\end{equation*}
\end{enumerate}
\end{definition}
\subsection{Main result}
Let us start by specifying the assumptions on the initial data. Let $\nu>\kappa$ and let $\eta$ be a small fixed positive number. We consider  an initial density $\rho^0$ such that 
\begin{equation}\label{eq:hyidr}
\begin{aligned}
&\rho^0\geq 0\textrm{ in }\T,\\
&\rho^0\in L^{1}\cap L^{\gamma}( \T),\\
&\nabla\rrho^0\in L^{2}\cap L^{2+\eta}( \T).
\end{aligned}
\end{equation}
Concerning the initial velocity $u_0$ we assume that 
\begin{equation}\label{eq:hyidu}
\begin{aligned}
&u_0=0\textrm{ on }\{\rho^0=0\},\\
&\sqrt{\rho^0}u^0\in L^{2}\cap L^{2+\eta}(\T).
\end{aligned}
\end{equation}
The hypothesis of higher integrability on $\nabla\sqrt{\rho^0}$ and $\sqrt{\rho^0}u^0$ imply that 
\begin{equation}\label{eq:mvi}
\rho^0\left(1+\frac{|v^0|^2}{2}\right)\log\left(1+\frac{|v^0|^2}{2}\right)\textrm{ is uniformly bounded in }L^{1}(\T),
\end{equation}
with $v^0=u^0+c\nabla\log\rho^0$ and $c>0$. In order to simplify the presentation we assume also that $\rho^0$ is bounded from above and below, namely there exists $\bar{\rho}^0>0$ such that
\begin{equation}\label{eq:hyidr2}
0<\frac{1}{\bar{\rho}^0}\leq\rho^0\leq \bar{\rho}^0.
\end{equation}
Then, we state our main result in the two dimensional case.  
\begin{theorem}\label{teo:main1}
Let $d=2$. Let $\nu, \kappa$ and $\g$ positive such that $\kappa<\nu$ and $\g>1$. Then for any $0<T<\infty$ there exists  a finite energy weak solutions of the system \eqref{eq:qns} on $(0, T)\times\mathbb{T}^2$, with initial data \eqref{eq:id} satisfying \eqref{eq:hyidr}, \eqref{eq:hyidu} and \eqref{eq:hyidr2}.
\end{theorem}
In the three dimensional case we need the a restriction on $\nu, \kappa$ and $\g$. 
\begin{theorem}\label{teo:main2}
Let $d=3$. Let $\nu, \kappa$ and $\g$ positive such that $\kappa^2<\nu^2<\frac{9}{8}\kappa^2$ and $1<\g<3$. Then for any $0<T<\infty$ there exists  a finite energy weak solutions of the system \eqref{eq:qns} on $(0, T)\times\mathbb{T}^3$, with initial data \eqref{eq:id} satisfying \eqref{eq:hyidr}, \eqref{eq:hyidu} and \eqref{eq:hyidr2}.
\end{theorem}
Let us briefly comment on the extra assumption we have in Theorem \ref{teo:main2}. This assumption is not required in the passage to the limit from the approximating solutions $(\re, \ue)$ to solutions of \eqref{eq:qns} but only in the proof of global existence of smooth solutions of the approximating system, see Theorem \ref{teo:main4}.  As it will be clear from our proof (see Proposition \ref{prop:ldl4}), we need the viscosity and capillarity constants to be comparable in order to prove regularity of solutions of the approximating system. The constant $9/8$ is not optimal there and can be improved. Furthermore we stress that we do not need any smallness assumptions on $\nu, \kappa$.
Recently, after the submission of our paper, it was shown in \cite{LLV} that the above technical restriction can indeed be removed. This was achieved by following some different arguments than the ones used in our paper.

\section{The Approximating System}\label{sec:app}
In this Section we first introduce the approximating system we are going to study and we then show how that can be transformed into an equivalent system in terms of the effective velocity, analogously to what was done in \cite{AS}.\\
\subsection{Approximating System}
The system in $(0,T)\times\T$  we consider is 
\begin{equation}\label{eq:aqns}
\begin{aligned}
&\partial_t\re+\dive(\re \ue)=0,\\
&\partial_t(\re \ue)+\dive(\re \ue\otimes \ue)-2\nu\dive\Sy_\eps+\nabla(\re^\g+p_{\e}(\re))+\tilde{p}_{\e}(\re)\ue=\kappa^2\dive\K_\eps.
\end{aligned}
\end{equation}
The system \eqref{eq:aqns} is coupled with initial data on $\{t=0\}\times\T$:
\begin{equation}\label{eq:aid}
\begin{aligned}
&\re(0,x)=\re^0(x),\\
&\re\ue(0,x)=\re^{0}(x)\ue^{0}(x).
\end{aligned}
\end{equation}
Let us describe in what follows the various terms appearing in \eqref{eq:aqns}.

 The viscosity coefficient $\he(\re)$ is defined as follows 
\begin{equation}\label{eq:hrho}
\he(\re)=\re+\e\re^{\frac{7}{8}}+\e\re^{\g}
\end{equation}
and we define $\gie(\re)$ to be
\begin{equation}\label{eq:grho}
\gie(\re)=\re \he'(\re)-\he(\re).
\end{equation}
Then the stress tensors $\Sy_\eps=\Sy_{\eps}(\nabla\ue)$ is: 
\begin{equation}\label{eq:avt}
\Sy_\eps(\nabla\ue)=\he(\re)D\ue+\gie(\re)\dive \ue\mathbb{I}.
\end{equation}
The following inequalities follow from the definitions of $\he(\re)$ and $\gie(\re)$
\begin{equation}\label{eq:stin1}
\begin{aligned}
&\he(\re)\geq0, &|\gie(\re)|\leq \max(\frac18, (\g-1))\he(\re),\\
&\he'(\re)\re\leq \g \he(\re), &|\he''(\re)|\re\leq (\g-1) \he'(\re).\\
\end{aligned}
\end{equation}
In particular it follows from \eqref{eq:hrho} that 
\begin{equation}\label{eq:stin2}
\begin{aligned}
&\he(\re)|D\ue|^2+\gie(\re)|\dive \ue|^2>\frac{5}{8}\he(\re)|D\ue|^2.
\end{aligned}
\end{equation}

The approximating dispersive term $\K_\e=\K_{\e}(\re,\nabla\re)$ is defined as 
\begin{equation*}
\dive(\K_{\e}(\re,\nabla\re))=2\re\nabla\left(\frac{\he'(\re)\dive(\he'(\re)\nabla\rre)}{\rre}\right).
\end{equation*}
We notice that, for $\eps=0$, we recover the quantum term in \eqref{eq:quantum}. Next Lemma clarifies how this approximation is consistent with the approximating viscous tensor in \eqref{eq:avt}.

\begin{lemma}\label{lem:formK}
The following formulae hold for the capillarity term $\diver \K_{\e}$: 
\begin{equation*}
\begin{aligned}
&2\re\nabla\left(\frac{\he'(\re)\dive(\he'(\re)\nabla\rre)}{\rre}\right)=\dive(\he(\re)\nabla^{2}\phie(\re))+\nabla(\gie(\re)\Delta\phie(\re))\\
&=\nabla\left(\he'(\re)\Delta \he(\re)\right)-4\diver((\he'(\re)\nabla\rre)\otimes (\he'(\re)\nabla\rre))\\
\end{aligned}
\end{equation*}
where $\phie(\re)$ is such that $\re\phi'(\re)=\he'(\re)$. 
\end{lemma}
\begin{proof}

By direct computations we get 
\begin{equation*}
\begin{aligned}
\dive(\K_\eps)&=\nabla(\he'(\re)\rre\dive(\he'(\re)\nabla\re/\rre))-4\he'(\re)\nabla\rre\dive(\he'(\re)\nabla\rre)\\
&=\nabla(\he'(\re)\Delta \he(\re))-2\nabla\left(|h'_\eps(\rho_\eps)\nabla\sqrt{\rho_\eps}|^2\right)\\
&-4\dive(\he'(\re)\nabla\rre\otimes \he'(\re)\nabla\rre)+4\nabla\left(h'_\eps(\rho_\eps)\nabla\sqrt{\rho_\eps}\right)\cdot(h'_\eps(\rho_\eps)\nabla\sqrt{\rho_\eps})\\
&=\nabla(\he'(\re)\Delta \he(\re))-4\dive(\he'(\re)\nabla\rre\otimes \he'(\re)\nabla\rre).
\end{aligned}
\end{equation*}        
To prove the remaining identity, we use the fact that $\re\phi'(\re)=\he'(\re)$ we have
\begin{equation*}
\begin{aligned}
&\nabla(\he'(\re)\Delta \he(\re))-4\diver(\he'(\re)\nabla\rre\otimes \he'(\re)\nabla\rre))=\\
&\nabla(\he'(\re)\dive(\re\nabla\phie(\re)))-\diver(\nabla \he(\re)\otimes\nabla\phie(\re))=\\
&\nabla(h'_\eps(\rho_\eps)\rho_\eps\Delta\phi_\eps(\rho_\eps))+\nabla(\nabla h_\eps(\rho_\eps)\cdot\nabla\phi_\eps(\rho_\eps))\\&-\nabla^2(h_\eps(\rho_\eps)\nabla\phi_\eps(\rho_\eps))+\diver(h_\eps(\rho_\eps)\nabla^2\phi_\eps(\rho_\eps))=\\
&\nabla(h'_\eps(\rho_\eps)\rho_\eps\Delta\phi_\eps(\rho_\eps))-\nabla(h_\eps(\rho_\eps)\Delta\phi_\eps(\rho_\eps))+\diver(h_\eps(\rho_\eps)\nabla^2\phi_\eps(\rho_\eps))=\\
&\nabla(g_\eps(\rho_\eps)\Delta\phi_\eps(\rho_\eps))+\diver(h_\eps(\rho_\eps)\nabla^2\phi_\eps(\rho_\eps)).
\end{aligned}
\end{equation*}
\end{proof}

The previous Lemma explains how the regularization of the dispersive tensor is consistent with \eqref{eq:hrho} and the transformation through the effective velocity. Indeed, since the viscous tensor $\mathbb{S}_{\e}(\nabla \ue)=\he(\re)D\ue+\gie(\re)\diver \ue$, the effective velocity is given by
$\ve=\ue+c\nabla\phie(\re)$, where as above $\phie(\re)$ is defined through $\he'(\re)=\re\phie'(\re)$. Then, from the identities in Lemma \ref{lem:formK}, it is straightforward to see that $\diver \mathbb K_\eps(\re,\nabla\re)=\diver \mathbb S_\eps(\nabla^2\phie(\re))$, so that in the effective system this can be incorporated in the effective viscous tensor. A similar Lemma is also proven in \cite{BCNV}, where the authors use an equivalent formula for the capillarity term in order to construct a numerical scheme for the Euler-Korteweg system with entropy stability property under a hyperbolic CFL condition.

The coefficient $\tilde{p}_{\e}(\re)$ in the damping term is defined by
\begin{equation*}
\tilde{p}_{\e}(\re)=\lambda(\e)\left(\re^{\frac{1}{\e^2}}+\re^{-\frac{1}{\e^2}}\right)
\end{equation*}
where $\lambda(\eps)=e^{-\frac{1}{\eps^4}}$. The cold pressure $p_\e(\re)$ is defined such that 
\begin{equation*}
p'_{\e}(\re)=\mu\tilde{p}_{\e}(\re)\frac{\he'(\re)}{\re},
\end{equation*}
where 
\begin{equation}\label{eq:maincost}
\mu=\nu-\sqrt{\nu^2-\kappa^2}.
\end{equation}
In particular, by using the definition of $\he(\re)$ and $\tilde{p}_{\e}(\re)$ by direct computations we get 
the following expression for $p_{\e}(\re)$ 
\begin{equation}\label{eq:defp}
\begin{aligned}
p_{\e}(\re)&=\mu\e^{2}\lambda(\e)\re^{\frac{1}{\e^2}}+\frac{\e^3\mu7\lambda(\e)}{8-\e^2}\re^{\frac{1}{\e^2}-\frac{1}{8}}\\
                  &+\frac{\e^3 \mu\lambda(\e)\gamma}{1+\e^2(\gamma-1)}\re^{\frac{1}{\e^2}+\gamma-1}-\mu\lambda(\e)\e^2\re^{-\frac{1}{\e^2}}\\
                  &-\frac{\e^3\mu7\lambda(\e)}{\e^2+8}\re^{-\frac{1}{\e^2}-\frac{1}{8}}-\frac{\e^3\mu\gamma\lambda(\e)}{1-\e^{2}(\gamma-1)}\re^{-\frac{1}{\e^2}+\gamma-1}\\
                  &=\sum_{i=1}^{6}p_{\e}^{i}(\re).
\end{aligned}
\end{equation}
Let $f_{\e}(\re)$ such that 
\begin{equation*}
p_{\e}(\re)=\re f'_{\e}(\re)-f_{\e}(\re).
\end{equation*}
Then, again by direct calculation we have that 
\begin{equation}\label{eq:deff}
\begin{aligned}
f_{\e}(\re)&=\frac{\mu\e^{4}\lambda(\e)}{1-\e^2}\re^{\frac{1}{\e^2}}+\frac{\e^5\mu7\lambda(\e)}{(8-\e^2)(8-9\e^2)}\re^{\frac{1}{\e^2}-\frac{1}{8}}\\
                 &+\frac{\e^5\mu\lambda(\e)\gamma}{(1+\e^2(\g-1))(1+\e^2(\g-2))}\re^{\frac{1}{\e^2}+\gamma-1}
                 +\frac{\e^2 \mu\lambda(\e)}{\e^2+1}\re^{-\frac{1}{\e^2}}\\
                 &+\frac{\e^5\mu7\lambda(\e)8}{(8+\eps^2)(9+8\e^2)}\re^{-\frac{1}{\e^2}-\frac{1}{8}}
                 +\frac{\e^5\mu\g\lambda(\e)}{(1-\e^2(\g-1))(1-\e^2(\g-2))}\re^{-\frac{1}{\e^2}+\gamma-1}
                 &=\sum_{i=1}^{6}f_{\e}^{i}(\re).
\end{aligned}
\end{equation}
It is straightforward to check that there exists $\e_{f}=\e_f(\g)>0$ small enough such that both $f^{i}_{\e}(\re)$ and $(f^{i}_{\e}(\re))''$ are positive for any $i=1,...,6$, $\eps<\eps_f$.\\

Finally we construct the initial data \eqref{eq:aid}. Given $(\rho^0, u^{0})$ satisfying \eqref{eq:hyidr}, \eqref{eq:hyidu} and \eqref{eq:hyidr2} it is easy to construct a sequence of smooth functions $(\re^0, \ue^{0})$ such that 
\begin{equation}\label{eq:hyaid}
\begin{aligned}
&\frac{1}{\bar{\rho}^0}\leq\re^0\leq \bar{\rho}^0,\\
&\re^0\rightarrow \rho^0\textrm{ strongly in }L^{1}(\T),\\
&\{\re^0\}_{\e}\textrm{ is uniformly bounded in }L^{1}\cap L^{\gamma}(\T),\\
&\{\he'(\re)\nabla\sqrt{\re^0}\}_\e\textrm{ is uniformly bounded in }L^{2}\cap L^{2+\eta}(\T),\\
& \he'(\re)^0\nabla\sqrt{\re^0}\rightarrow\nabla\sqrt{\rho^0}\textrm{ strongly in }L^{2}(\T),\\
&\{\sqrt{\re^0}\ue^0\}\textrm{ is uniformly bounded in }L^{2}\cap L^{2+\eta}(\T),\\
&\re^0 \ue^0\rightarrow \re^0 \ue^0\textrm{ in }L^{1}(\T),\\ 
&f_{\e}(\re^0)\rightarrow 0\textrm{ strongly in }L^{1}(\T).
\end{aligned}
\end{equation}
In particular the hypothesis on the boundedness of $\rho^0$ makes it easy to prove that $h'(\re)$ and $f_{\e}(\re)$ are uniformly bounded. Moreover, the higher integrability on $h'(\re^{0})\nabla\sqrt{\re^0}$ and $\sqrt{\re^0}\ue^0$ implies that 
\begin{equation}\label{eq:mvi}
\re^0\left(1+\frac{|v_\eps^0|^2}{2}\right)\log\left(1+\frac{|v_\eps^0|^2}{2}\right)\textrm{ is uniformly bounded in }L^{1}(\T),
\end{equation}
with $v_\eps^0=u^0+c\nabla\phi_\eps(\re^0)$ and $c>0$.\\
\subsection{The effective velocity formulation}
We now consider the effective velocity $v_\eps=u_\eps+c\nabla\phi_\eps(\rho_\eps)$. The next Lemma shows that the system \eqref{eq:aqns} can be equivalently written in terms of $(\rho_\eps, v_\eps)$. Furthermore, with a suitable choice of the constant $c$, both the dispersive and the cold pressure terms will vanish.
\begin{lemma}\label{lem:BDtransf}
Let $(\re,\ue)$ be a smooth solution of the system \eqref{eq:aqns}. Then, $(\re,\ve)$, with $\ve=\ue+c\nabla\phie(\re)$ and $c>0$ satisfies the following system,
\begin{equation}\label{eq:avqns}
\begin{aligned}
&\d_t\re+\diver(\re \ve)=c\Delta \he(\re)\\
&\d_t(\re \ve)+\diver(\re \ve\otimes \ve)+\nabla \re^{\gamma}+\tilde\lambda\nabla p_\eps(\rho_\eps)-c\Delta(\he(\re) \ve)+\tilde{p}(\re)\ve\\
&-2(\nu-c)\dive(\he(\re)D\ve)-(2\nu-c)\nabla(\gie(\re)\dive \ve)-\tilde\kappa^2\diver\K_\eps=0,
\end{aligned}
\end{equation}
where $\mu>0$ is defined in \eqref{eq:maincost}, $\tilde\kappa^2=\kappa^2-2\nu c+c^2$, $\tilde\lambda=(\mu-c)/\mu$.
\end{lemma}
\begin{proof}
Let $c\in\mathbb{R}$. From the first equation in \eqref{eq:aqns} we have that 
\begin{equation}\label{eq:id1}
c(\re\nabla\phie(\re))_t=-c\nabla(\dive(\he(\re)\ue))-c\nabla(\gie(\re)\dive \ue).
\end{equation}
Moreover, it is straightforward to prove that
\begin{equation}\label{eq:id2}
\begin{aligned}
c\dive(\re \ue\otimes\nabla\phie(\re)+\re\nabla\phie(\re)\otimes \ue)&=c\Delta(\he(\re) \ue)-2c\dive(\he(\re) D\ue)\\
&+c\nabla\dive(\he(\re) \ue)\\
\end{aligned}
\end{equation}
and 
\begin{equation}\label{eq:id3}
\begin{aligned}
c^{2}\dive(\re\nabla\phie(\re)\otimes\nabla\phie(\re))&=c^2\Delta(\he(\re)\nabla\phie(\re))-c^2\dive(\he(\re)\nabla^2\phie(\re)),
\end{aligned}
\end{equation}
see also \cite{J3}. Then, by using the definition of $\ve$ we have 
\begin{equation}
\begin{aligned}
&\partial_{t}(\re \ve)+\dive(\re \ve\otimes \ve)+\nabla\re^{\g}=\\
&\partial_{t}(\re \ue)+\dive(\re \ue\otimes \ue)+\nabla\re^{\g}\\
&+c(\re\nabla\phie(\re))_t+c\dive(\re \ue\otimes\nabla\phie(\re)+\re\nabla\phie(\re)\otimes \ue)\\
&+c^2\Delta(\he(\re)\nabla\phie(\re))-c^2\dive(\he(\re)\nabla^2\phie(\re))
\end{aligned}
\end{equation}
and by using \eqref{eq:id1}-\eqref{eq:id3} and the fact the $(\re,\ue)$ satisfies the momentum equation in \eqref{eq:aqns} we get 
\begin{equation}
\begin{aligned}
&\d_t(\re \ve)+\diver(\re \ve\otimes \ve)-c\Delta(\he(\re) \ve)+\nabla \re^{\gamma}\\
&-2(\nu-c)\dive(\he(\re)D\ve)-(2\nu-c)\nabla(\gie(\re)\dive \ve)+\tilde{p}(\re)\ve\\
&=c\tilde{p}_{\e}(\re)\nabla\phie(\re)-p'_{\e}(\re)\nabla\re+(\kappa^2-2\nu c+c^2)\diver\mathbb K_\eps(\rho_\eps).
\end{aligned}
\end{equation}
By using that $\ve=\ue+c\nabla\phie(\re)$, Lemma \ref{lem:formK} and the definition on $p_{\e}(\re)$ we get 
\begin{equation}\label{eq:320}
\begin{aligned}
&\d_t(\re \ve)+\diver(\re \ve\otimes \ve)-c\Delta(\he(\re) \ve)+\nabla \re^{\gamma}\\
&+c\dive(\he(\re)D\ve)-(2\nu-c)\dive\Sy_\eps(\ve)+\tilde{p}(\re)\ve\\
&=(c^{2}-2\nu c+\kappa^2)\dive\K_\eps(\re)-\frac{\mu-c}{\mu}\nabla p_{\e}(\re).
\end{aligned}
\end{equation}
Let us notice that, by taking $c=\mu$, then the right hand side in \eqref{eq:320} vanishes.
 \end{proof}

\section{A priori Estimates}\label{sec:apriori}
In this Section we are going to show that the approximating system satisfies, uniformly in $\eps>0$, the a priori estimates used in \cite{AS} to prove the compactness of weak solutions to \eqref{eq:qns}. First of all we prove the classical energy estimate for system \eqref{eq:aqns}.
\begin{proposition}\label{prop:energy}
Let $(\re,\ue)$ be a smooth solution of \eqref{eq:aqns}. Then, the following estimate holds.
\begin{equation}\label{eq:energy}
\begin{aligned}
&\frac{d}{dt}\left(\int\frac{\re|\ue|^2}{2}+\frac{\re^{\g}}{\g-1}+f_{\e}(\re)+2\kappa^2|\he'(\re)\nabla\rre|^2\right)\\
&+2\nu\int \he(\re)|D\ue|^2+2\nu\int \gie(\re)|\dive \ue|^2+\int\tilde{p}_{\e}(\re)|\ue|^2=0.
\end{aligned}
\end{equation}
\end{proposition}
\begin{proof}
Let us multiply the momentum equation in \eqref{eq:aqns} by $\ue$. After integrating by parts and using the first equation we get 
\begin{equation}\label{eq:energy1}
\begin{aligned}
&\frac{d}{dt}\int\frac{\re|\ue|^2}{2}+2\nu\int \he(\re)|D\ue|^2+2\nu\int \gie(\re)|\dive \ue|^2+\int\tilde{p}_{\e}(\re)|\ue|^2\\
&-\kappa^2\int\dive\K_{\e}\ue+\int\nabla(\re^{\g}+p_{\e}(\re))\ue=0.
\end{aligned}
\end{equation}
Then, we consider the pressure terms. By multiplying the first equation by $\frac{\g\re^{\g-1}}{\g-1}$ we get 
\begin{equation}\label{eq:energy2}
\frac{d}{dt}\int\frac{\re^\g}{\g-1}-\int\nabla\re^\g \ue=0.
\end{equation}
By multiplying again the first equation by $f'_{\e}(\re)$  
\begin{equation}\label{eq:energy3}
\frac{d}{dt}\int f_{\e}(\re)-\int\nabla p_{\e}(\re)\ue=0.
\end{equation} 
Finally, we deal with the dispersive term. By multiplying the first equation by\\
 $-2\kappa^2\he'(\re)\dive(\he'(\re)\nabla\rre)/\rre$ we get 
\begin{equation*}
\begin{aligned}
-2\kappa^2\int\partial_t\re\frac{\he'(\re)\dive(\he'(\re)\nabla\rre)}{\rre}&-2\kappa^2\int\dive(\re \ue)\frac{\he'(\re)\dive(\he'(\re)\nabla\rre)}{\rre}=0.\\
\end{aligned}
\end{equation*}
Then, by using Lemma \ref{lem:formK}, integrating by parts and using the chain rule we get 
\begin{equation}\label{eq:energy4}
\frac{d}{dt}\int2\kappa^2|\he'(\re)\nabla\rre|^2+\kappa^2\int\dive\K_{\e} \ue=0.
\end{equation}
By summing up \eqref{eq:energy1}, \eqref{eq:energy2}, \eqref{eq:energy3} and \eqref{eq:energy4} we get \eqref{eq:energy}.
\end{proof}
Next Lemma gives the energy estimate for the transformed system \eqref{eq:avqns}. 
\begin{proposition}\label{prop:bdentropy}
Let $(\rho_\eps, u_\eps)$ be a smooth solution to \eqref{eq:aqns} and let us consider $(\rho_\eps, v_\eps)$, where $v_\eps$ is the effective velocity $v_\eps=u_\eps+c\nabla\phi_\eps(\rho_\eps)$, with $c\in(0, \mu)$. Then we have
\begin{equation}\label{eq:bd}
\begin{aligned}
&\frac{d}{dt}\left(\int\frac{\re|\ve|^2}{2}+\frac{\re^{\g}}{\g-1}+\tilde\lambda f_{\e}(\re)+2\tilde{\kappa}^2|\he'(\re)\nabla\rre|^2\right)\\
&+c\int \he(\re)|A\ve|^2+(2\nu-c)\int( \he(\re)|D\ve|^2+\gie(\re)|\dive \ve|^2)\\
&+\int\tilde{p}_{\e}(\re)|\ve|^2+c\g\int \he'(\re)|\nabla\re|^2\re^{\g-2}+c\tilde{\lambda}\int \he'(\re)|\nabla\re|^2f_{\e}''(\re)\\
&+c\tilde{\kappa}^2\int \he(\re)|\nabla^{2}\phie(\re)|^2+c\tilde{\kappa}^2\int \gie(\re)|\Delta\phie(\re)|^2=0,
\end{aligned}
\end{equation}
where $\tilde{\lambda}=(\mu-c)/\mu$, $\tilde{\kappa}^2=c^{2}-2\nu c+\kappa^2$ and we recall that $A\ve=(\nabla \ve-(\nabla \ve)^T)/2$ is the antisymmetric part of the gradient.
\end{proposition}
\begin{proof}
Since $(\re,\ue)$ is a smooth solution of \eqref{eq:aqns} we can use Lemma \ref{lem:BDtransf} to deduce that $(\re, \ve)$ satisfies equations \eqref{eq:avqns}. Then, by multiplying the momentum equation by $\ve$, integrating by parts and using the first equation we get 
\begin{equation}\label{eq:bdentropy1}
\begin{aligned}
&\frac{d}{dt}\int \frac{\re|\ve|^2}{2}+c\int \he(\re)|A\ve|^2+(2\nu-c)\int (\he(\re)|D\ve|^2+ \gie(\re)|\dive \ve|^2)\\
&+\int\nabla\re^{\g}\cdot \ve+\tilde{\lambda}\int\nabla p_{\e}(\re)\cdot \ve+\int\tilde{p}_{\e}(\re)|\ve|^2-\tilde{\kappa}^2\int\dive\K_{\e}\ve=0,
\end{aligned}
\end{equation}
where we used that $|\nabla\ve|^2=|D\ve|^2+|A\ve|^2$. 
Then, by multiplying the first equation by $\frac{\g\re^{\g-1}}{\g-1}$ and integrating by parts we get 
\begin{equation}\label{eq:bdentropy2}
\begin{aligned}
\frac{d}{dt}\int\frac{\re^{\gamma}}{\gamma-1}+c\g\int|\nabla\re|^2\he'(\re)\re^{\g-2}-\int\nabla\re^{\g}\ve=0.
\end{aligned}
\end{equation}
By multiplying again the first equation by $\tilde{\lambda}f'_{\e}(\re)$ we get
\begin{equation}\label{eq:bdentropy3}
\frac{d}{dt}\int \tilde{\lambda}f_{\e}(\re)-\tilde{\lambda}\int\nabla p_{\e}(\re)\ve+c\tilde{\lambda}\int \he'(\re)|\nabla\re|^2f_{\e}''(\re)=0.
\end{equation} 
Then, we consider the dispersive term. By multiplying the first equation\\ by $-2\tilde{\kappa}^2\he'(\re)\dive(\he'(\re)\nabla\rre)/\rre$ we get 
\begin{equation*}
\begin{aligned}
\\
&-2\tilde{\kappa}^2\int\partial_t\re\frac{\he'(\re)\dive(\he'(\re)\nabla\rre)}{\rre}-2\tilde{\kappa}^2\int\dive(\re \ve)\frac{\he'(\re)\dive(\he'(\re)\nabla\rre)}{\rre}\\
&+c2\tilde{\kappa}^2\int\Delta \he(\re)\frac{\he'(\re)\dive(\he'(\re)\nabla\rre)}{\rre}=0.
\end{aligned}
\end{equation*}
The first two terms are treated as in Proposition \ref{prop:energy} and we get 
\begin{equation*}
\begin{aligned}
&-2\tilde\kappa^2\int\d_t\rho_\eps\frac{h'_\eps(\rho_\eps)\diver(h'_\eps(\rho_\eps)\nabla\sqrt{\rho_\eps})}{\sqrt{\rho_\eps}}-2\tilde\kappa^2\int\diver(\rho_\eps v_\eps)\frac{h'_\eps(\rho_\eps)\diver(h'_\eps(\rho_\eps)\nabla\sqrt{\rho_\eps})}{\sqrt{\rho_\eps}}\\
&\quad=\frac{d}{dt}\int2\tilde\kappa^2|h'_\eps(\rho_\eps)\nabla\sqrt{\rho_\eps}|^2-\tilde{\kappa}^2\int\dive\K_{\e}\ve.
\end{aligned}
\end{equation*}
We then consider the last term,  by integrating by parts  we get
\begin{equation*}
\begin{aligned}
2\int\Delta \he(\re)\frac{\he'(\re)\dive(\he'(\re)\nabla\rre)}{\rre}=&-2\int \nabla\phie(\re)\re\nabla\left(\frac{\he'(\re)\dive(\he'(\re)\nabla\rre)}{\rre}\right)\\
=&-\int\nabla\phie(\re)\dive\K_{\e}\\
=&-\int\nabla\phie(\re)\dive(\he(\re)\nabla^{2}\phie(\re)\\
&-\int\nabla\phie(\re)\nabla(\gie(\re)\Delta\phie(\re),
\end{aligned}
\end{equation*}
where Lemma \ref{lem:formK} has been used. By integrating by parts we get 
\begin{equation*}
2\int\Delta \he(\re)\frac{\he'(\re)\dive(\he'(\re)\nabla\rre)}{\rre}=\int \he(\re)|\nabla^{2}\phie(\re)|^2+\int \gie(\re)|\Delta\phie(\re)|^2.
\end{equation*}
Resuming, we have
\begin{equation}\label{eq:bdentropy4}
\begin{aligned}
&\frac{d}{dt}\int2\tilde{\kappa}^2|\he'(\re)\nabla\rre|^2+\tilde{\kappa}^2\int\dive\K(\re) \ve\\
&+c\tilde{\kappa}^2\int \he(\re)|\nabla^{2}\phie(\re)|^2+c\tilde{\kappa}^2\gie(\re)|\Delta\phie(\re)|^2=0
\end{aligned}
\end{equation}
By summing up \eqref{eq:bdentropy1}, \eqref{eq:bdentropy2}, \eqref{eq:bdentropy3} and \eqref{eq:bdentropy4} we get \eqref{eq:bd}
\end{proof}
Let us now choose the constant in the effective velocity to be exactly the constant defined in \eqref{eq:maincost}, $\mu=\nu-\sqrt{\nu^2-\kappa^2}$. Throughout this paper we will denote by $w^\eps$ the effective velocity with this particular choice of the constant, i.e. $w_\eps=u_\eps+\mu\nabla\phi_\eps(\rho_\eps)$. As we already noticed, in this case both the dispersive term and the cold pressure term vanish in \eqref{eq:avqns}, so that the system reads
\begin{equation}\label{eq:awqns}
\begin{aligned}
&\d_t\re+\diver(\re \we)=\mu\Delta \he(\re),\\
&\d_t(\re \we)+\diver(\re \we\otimes \we)-\mu\Delta(\he(\re) \we)+\nabla \re^{\gamma} +\tilde{p}(\re)\we\\
&-2(\nu-\mu)\dive(\he(\re)D\we)-(2\nu-\mu)\nabla(\gie(\re)\dive\we)=0.
\end{aligned}
\end{equation}
Analogously to what we did in \cite{AS}, we now prove a Mellet-Vasseur type estimate for \eqref{eq:awqns}. We will first prove an auxiliary Lemma which will also be useful later in section \ref{sec:appex}, see Lemma \ref{prop:ldl4}.
\begin{lemma}\label{lem:renormalized}
Let $(\re, \ue)$ be a solution of the system \eqref{eq:aqns}. Then, for any $\beta\in C^{1}(\R)$ the pair $(\re, \we)$ satisfies the following integral equation
\begin{equation}\label{eq:renormalization}
\begin{aligned}
&\frac{d}{dt}\int\re\be+\mu\int \he(\re)|A\we\cdot\we|^2\bedp+\mu\int \he(\re)|A \we|^2\bep\\
&+(2\nu-\mu)\int \he(\re)|D \we|^2\bep+(2\nu-\mu)\int \gie(\re)|\dive \we|^2\bep\\
&+\int\tilde{p}_{\e}(\re)|\we|^2\bep+(2\nu-\mu)\int \he(\re)|D \we\cdot \we|^2\bedp\\
&=-\int\nabla\re^{\g}\we\bep-2\nu\int \he(\re)(D\we\cdot\we)\cdot(A\we\cdot\we)\bedp\\
&-(2\nu-\mu)\int \gie(\re)\dive \we\we\cdot(D\we\cdot\we)\bedp.
\end{aligned}
\end{equation}
\end{lemma}
\begin{proof}
Let $\beta\in C^{1}(\R)$. By a simple integration by parts we get 
\begin{equation*}
\begin{aligned}
-\mu\int\Delta(\he(\re)\we)\we\bep=&-\mu\int\Delta \he(\re)|\we|^2\bep\\
                                              &+\mu\int\Delta \he(\re)\be\\
                                              &+\mu\int \he(\re)|\nabla \we|^2\bep\\
                                             &+\mu\int \he(\re)\left|\nabla\frac{|\we|^2}{2}\right|^2\bedp.
\end{aligned}                                             
\end{equation*}
Then, by multiplying the second equation in \eqref{eq:awqns} by $\we\bep$ we get 
\begin{equation}\label{eq:renormalization1}
\begin{aligned}
&\frac{d}{dt}\int\re\be+\mu\int \he(\re)\left|\nabla\left(\frac{|\we|^2}{2}\right)\right|^2\bedp\\
&2(\nu-\mu)\int \he(\re)|D \we|^2\bep+(2\nu-\mu)\int \gie(\re)|\dive \we|^2\bep\\
&+\int\tilde{p}_{\e}(\re)|\we|^2\bep+2(\nu-\mu)\int D\we \we\nabla \we \we\bedp\\
&=-\int\nabla\re^{\g}\we\bep-(2\nu-\mu)\int \gie(\re)\dive \we\we\nabla\left(\frac{|\we|^2}{2}\right)\bedp.\\
\end{aligned}
\end{equation}
Let us consider the last term on the left-hand side of the equality, we have 
\begin{equation}\label{eq:renormalization2}
\begin{aligned}
&2(\nu-\mu)\int\he(\re) \left(\frac{\partial_j \we^i+\partial_i \we^j}{2}\right) w^i\partial_j \we^l \we^l\bedp=\\
&2(\nu-\mu)\int\he(\re) \left(\frac{\partial_j \we^i+\partial_i \we^j}{2}\right) \we^i\left(\frac{\partial_j \we^l+\partial_l\we^j}{2}\right) \we^l\bedp+\\
&2(\nu-\mu)\int\he(\re) \left(\frac{\partial_j \we^i+\partial_i \we^j}{2}\right) \we^i\left(\frac{\partial_j \we^l-\partial_l\we^j}{2}\right) \we^l\bedp.
\end{aligned}
\end{equation}
Concerning the last term in the right-hand side of the inequality we have 
\begin{equation}\label{eq:renormalization3}
\begin{aligned}
&\int \gie(\re)\dive \we\we^i\partial_i\we^l\we^l\bedp\\
&=\int \gie(\re)\dive \we \we^i\left(\frac{\partial_i\we^l+\partial_l\we^i}{2}\right)\we^l\bedp\\
&+\int \gie(\re)\dive \we \we^i\left(\frac{\partial_i\we^l-\partial_l\we^i}{2}\right)\we^l\bedp\\
&=\int \gie(\re)\dive \we \we^i\left(\frac{\partial_i\we^l+\partial_l\we^i}{2}\right)\we^l\bedp.
\end{aligned}
\end{equation}
Finally, by using that 
\begin{equation}\label{eq:renomarlized4}
\begin{aligned}
&\int\he(\re)\left|\nabla\left(\frac{|\we|^2}{2}\right)\right|^2\bedp\\
&=\int\he(\re)|A\we\cdot \we|^2\bedp\\
&+\int\he(\re)|D\we\cdot \we|^2\bedp\\
&+2\int\he(\re)(A\we\cdot\we)(D\we\cdot\we)\bedp.
\end{aligned}
\end{equation}

Then, by using \eqref{eq:renormalization2}, \eqref{eq:renormalization3} and \eqref{eq:renomarlized4}, we get from \eqref{eq:renormalization1} exactly \eqref{eq:renormalization}.
\end{proof}
Now, we are in position to prove the Mellet \& Vasseur type inequality. 
\begin{proposition}\label{prop:mvinequality}
Let $(\re, \ue)$ be a smooth solution of \eqref{eq:aqns}. Then, there exists and a generic constant $C>0$ independent on $\e$ such that $(\re,\we)$ sastifies
\begin{equation}\label{eq:mvinequality}
\begin{aligned}
&\sup_{t\in (0,T)}\int\re\left(1+\frac{|\we|^2}{2}\right)\log\left(1+\frac{|\we|^2}{2}\right)\leq C\iint \he(\re)|\nabla\we|^2\\
&+\int\!\!\left(\int\re^{(2\gamma-\delta/2-1)(2/(2-\delta))}dx\right)^{\frac{2-\delta}{2}}\!\!\left(\int\re \left(1+\log\left(1+\frac{|\we|^2}{2}\right)\right)^{\frac{2}{\delta}}\!dx\right)\!dt\\
&+\int\re^0\left(1+\frac{|\we^0|^2}{2}\right)\log\left(1+\frac{|\we^0|^2}{2}\right)
\end{aligned}
\end{equation}
for any $\delta\in(0,2)$.
\end{proposition}
\begin{proof}
By choosing $\beta(t)=(1+t)\log(1+t)$ in Lemma \ref{lem:renormalized} and keeping only the terms we need we get that there exists a generic constant $C>0$ 
independent on $\e$ such that
\begin{equation*}
\begin{aligned}
&\sup_{t\in (0,T)}\int\re\left(1+\frac{|\we|^2}{2}\right)\log\left(1+\frac{|\we|^2}{2}\right)+\iint\re|\nabla \we|^2\log\left(1+\frac{|\we|^2}{2}\right)\\
&\leq C\left|\int\nabla\re^{\g}\we\bep\right|+C\int \he(\re)|\nabla\we|^2+C\int |\gie(\re)||\dive \we||\nabla\we|\\
&+\int\re^0\left(1+\frac{|\we^0|^2}{2}\right)\log\left(1+\frac{|\we^0|^2}{2}\right).
\end{aligned}
\end{equation*}
Then, by using \eqref{eq:stin1}, integrating by parts the first term, using H\"older and Young inequality we get 
\begin{equation*}
\begin{aligned}
&\sup_{t\in (0,T)}\int\re\left(1+\frac{|\we|^2}{2}\right)\log\left(1+\frac{|\we|^2}{2}\right)+\iint\re|\nabla \we|^2\log\left(1+\frac{|\we|^2}{2}\right)\\
&\leq C\iint\re^{2\g-1}\left((1+\log\left(1+\frac{|\we|^2}{2}\right)\right)+\frac{1}{2}\iint\re|\nabla \we|^2\log\left(1+\frac{|\we|^2}{2}\right)\\
&+C\iint \he(\re)|\nabla\we|^2+\int\re^0\left(1+\frac{|\we^0|^2}{2}\right)\log\left(1+\frac{|\we^0|^2}{2}\right).
\end{aligned}
\end{equation*}
Finally, for $\delta\in(0,2)$ by using H\"older inequality we get 
\begin{equation}\label{eq:mv5}
\begin{aligned}
&\iint\re^{2\gamma-1}\left(1+\log\left(1+\frac{|\we|^2}{2}\right)\right)\\
&\leq\int\!\!\left(\int\re^{(2\gamma-\delta/2-1)(2/(2-\delta))}dx\right)^{\frac{2-\delta}{2}}\!\!\left(\int\re \left(1+\log\left(1+\frac{|\we|^2}{2}\right)\right)^{\frac{2}{\delta}}\!dx\right)\!dt\\
&\leq C\int\!\!\left(\int\re^{(2\gamma-\delta/2-1)(2/(2-\delta))}\,dx\right)^{(2-\delta)/2}\,dt.
\end{aligned}
\end{equation}
Then, \eqref{eq:mvinequality} is proved.  
\end{proof}
\section{Proof of Theorem \ref{teo:main1} and Theorem \ref{teo:main2}}\label{sec:proof}
In this section prove we give the proofs of Theorem  \ref{teo:main1} and Theorem \ref{teo:main2}. Let us start by collecting and deriving the main bounds which will be needed.\\

\subsection{Uniform Bounds} 
Let $\e<\e_f$ and let $\{(\re,\ue)\}_\e$, with $\re>0$, be a sequence of smooth solutions of \eqref{eq:aqns} with initial data $(\re^0,\ue^0)$ satisfying \eqref{eq:hyaid}. The global existence of $(\re,\ue)$ will be proved in the next section.  
By Proposition \ref{prop:energy} there exists a generic constant $C>0$ independent on $\e$ such that   
\begin{equation}\label{eq:uf1}
\begin{aligned}
\\
\sup_{t}\int\rho_\e|u_\e|^2\leq C,\quad &\;\sup_{t}\int|\he'(\re)\nabla\sqrt{\rho_\e}|^2\leq C,\\
\\
\sup_{t}\int(\rho_\e+\rho^{\g}_{\e})\leq C,\quad &\; \iint \he(\re)|D u_\e|^{2}\leq C,\\
\\
\sup_{t}\int f_{\e}(\re)\leq C,\quad&\;\iint |\tilde{p}_{\e}(\re)||\ue|^2\leq C.\\
\\
\end{aligned}
\end{equation}
where \eqref{eq:stin1} has been used. In particular, by using \eqref{eq:hrho} we have that 
\begin{equation}\label{eq:uf2}
\begin{aligned}
\sup_{t}\int|\nabla\sqrt{\rho_\e}|^2\leq C,\quad &\; \iint \rho_\e|D u_\e|^{2}\leq C.\\
\end{aligned}
\end{equation}
Then, by \eqref{eq:hyaid} and Proposition \ref{prop:bdentropy} we get that there exists a generic constants $C>0$ independent on $\e$ such that 
\begin{equation}\label{eq:uf3}
\begin{aligned}
\\
\iint \he(\re)|A u_\e|^{2}\leq C,\quad &\;\iint \he'(\re)|\nabla\re|^2\re^{\g-2}\leq C,\\
\\
\iint \he'(\re)|\nabla\re|^2f_{\e}''(\rho)\leq C,\quad &\; \iint \he(\re)|\nabla^{2}\phie(\re)|^2\leq C,\\
\\
\end{aligned}
\end{equation}
where we have used \eqref{eq:stin1} and the fact that $A\we=A\ue$. 
In particular combining \eqref{eq:uf1}, \eqref{eq:uf2}, \eqref{eq:uf3} and \eqref{eq:hrho} we have 
\begin{equation}\label{eq:uf4}
\begin{aligned}
\iint \he(\rho_\e)|\nabla u_\e|^{2}\leq C,\quad &\;\iint\rho_\e|\nabla u_\e|^{2}\leq C.
\end{aligned}
\end{equation}
Next we consider the pressure. From \eqref{eq:uf1} and \eqref{eq:uf3}, after using \eqref{eq:hrho}, we get 
\begin{equation}\label{eq:uf5bis}
\iint|\nabla\re^{\frac{\g}{2}}|^2\leq C
\end{equation}
and then by interpolation with \eqref{eq:uf1} we have that 
\begin{equation}\label{eq:uf5}
\iint\re^{\frac{5\g}{3}}\leq C\textrm{ for any $\g>1$ if $d=2$ and for any $\g\in(1,3)$ if $d=3$}.
\end{equation}
Finally, we have the following uniform bounds
\begin{equation}\label{eq:uf6}
\iint|\nabla^{2}\rrho_\eps|^2+|\nabla\rho_\eps^{\frac{1}{4}}|^4\leq C.
\end{equation}
The bounds \eqref{eq:uf6} are crucial to handle the passage to the limit in the dispersive term and are not a straightforward consequence of the  a priori estimates. Indeed in order to obtain them we need a generalization of the inequality 
\begin{equation*}
\iint|\nabla^{2}\rrho|^2+|\nabla\rho^{\frac{1}{4}}|^4\leq C \iint\rho|\nabla^2\log\rho|^2
\end{equation*}
proved in \cite{J}, see also \cite{VY2} for an alternative proof.
\begin{lemma}\label{lem:dispest}
Let $\rho>0$ and $h(\rho)$ be a smooth function such that 
\begin{equation}\label{eq:lem51h}
\begin{aligned}
&h(\rho)\geq0,&h'(\rho)>0,&\quad h'(\rho)\rho\leq Ch(\rho),&\quad |h''(\rho)|\rho\leq Ch'(\rho)
\end{aligned}
\end{equation} 
Then, the following inequality hold 
\begin{equation}\label{eq:lem51a}
\iint h'(\rho)|\nabla(h'(\rho)\nabla\rrho)|^2+\iint\frac{(h'(\rho))^3|\nabla\sqrt{\rho}|^4}{\rho}\leq C\iint h(\rho)|\nabla^2\phi(\rho)|^2\\
\end{equation}
where $\rho \phi'(\rho)=h'(\rho)$. Moreover, if in addition we assume that $h'(\rho)\geq c>0$ then
\begin{equation}\label{eq:lem51b}
\iint |\nabla^2\rrho|^2+\iint|\nabla\rho^{\frac{1}{4}}|^4\leq C\iint h(\rho)|\nabla^2\phi(\rho)|^2.
\end{equation}
\end{lemma}
\begin{proof}
By using \eqref{eq:lem51h} we have that 
\begin{equation}\label{eq:sd2}
\iint \rho h'(\rho)\left|\nabla\left(\frac{h'(\rho)\nabla\sqrt{\rho}}{\sqrt{\rho}}\right)\right|^2=\iint \rho h'(\rho)|\nabla^2\phi(\rho)|^2\leq C \iint h(\rho)|\nabla^2\phi(\rho)|^2,
\end{equation}
where $h'(\rho)=\rho\phi'(\rho)$ has been used. 
Then, by using the chain rule we have
\begin{equation}\label{eq:sd4}
\nabla\left(\frac{h'(\rho)\nabla\sqrt{\rho}}{\sqrt{\rho}}\right)=\frac{\nabla(h'(\rho)\nabla\sqrt{\rho})}{\sqrt{\rho}}-\frac{h'(\rho)\nabla\sqrt{\rho}\otimes\nabla\sqrt{\rho}}{\rho}.
\end{equation}
Taking the square we get 
\begin{equation}\label{eq:sd5}
\begin{aligned}
\left|\nabla\left(\frac{h'(\rho)\nabla\sqrt{\rho}}{\sqrt{\rho}}\right)\right|^2&=\frac{|\nabla(h'(\rho)\nabla\sqrt{\rho})|^2}{\rho}\\
&-2\frac{\nabla(h'(\rho)\nabla\sqrt{\rho})h'(\rho)\nabla\sqrt{\rho}\otimes\nabla\sqrt{\rho}}{\sqrt{\rho}\rho}\\
&+\frac{(h'(\rho))^2|\nabla\sqrt{\rho}|^4}{\rho^2}.
\end{aligned}
\end{equation}
Then, by using \eqref{eq:sd2}
\begin{equation}
\label{eq:sd6}
\begin{aligned}
&\iint h'(\rho)|\nabla(h'(\rho)\nabla\sqrt{\rho})|^2+\frac{(h'(\rho))^3|\nabla\sqrt{\rho}|^4}{\rho}\leq C\iint h(\rho)|\nabla^2\phi(\rho)|^2\\ 
&+2\iint\rho h'(\rho)\frac{\nabla(h'(\rho)\nabla\sqrt{\rho})h'(\rho)\nabla\sqrt{\rho}\otimes\nabla\sqrt{\rho}}{\sqrt{\rho}\rho}
\end{aligned}
\end{equation}
Now we focus on the last term of \eqref{eq:sd6}. By integrating by parts we get
\begin{equation*}
\begin{aligned}
&\iint h'(\rho)\frac{\nabla(h'(\rho)\nabla\sqrt{\rho})h'(\rho)\nabla\sqrt{\rho}\otimes\nabla\sqrt{\rho}}{\sqrt{\rho}}=\\
&\iint\frac{\partial_i (h'(\rho)\partial_j\sqrt{\rho})h'(\rho)\partial_i\sqrt{\rho}h'(\rho)\partial_j\sqrt{\rho}}{\sqrt{\rho}}=\\
&-\iint\frac{\partial_i (h'(\rho)\partial_j\sqrt{\rho})h'(\rho)\partial_i\sqrt{\rho}h'(\rho)\partial_j\sqrt{\rho}}{\sqrt{\rho}}\\
&-\iint h'(\rho)\partial_j\sqrt{\rho}\partial_i\left(\frac{h'(\rho)\partial_i\sqrt{\rho}}{\sqrt{\rho}}\right)h'(\rho)\partial_j\sqrt{\rho}.
\end{aligned}
\end{equation*}
Then, 
\begin{equation}\label{eq:sd8}
\begin{aligned}
&2\iint h'(\rho)\frac{\nabla(h'(\rho)\nabla\sqrt{\rho})h'(\rho)\nabla\sqrt{\rho}\otimes\nabla\sqrt{\rho}}{\sqrt{\rho}}=\\
&-\iint h'(\rho)\nabla\sqrt{\rho}\dive\left(\frac{h'(\rho)\nabla\sqrt{\rho}}{\sqrt{\rho}}\right)h'(\rho)\nabla\sqrt{\rho}=\\
&-\iint \sqrt{\rho}\sqrt{h'(\rho)}\frac{\sqrt{h'(\rho)}}{\sqrt{\rho}}\nabla\sqrt{\rho}\dive\left(\frac{h'(\rho)\nabla\sqrt{\rho}}{\sqrt{\rho}}\right)h'(\rho)\nabla\sqrt{\rho}\leq\\
&C \iint \rho h'(\rho)\left|\nabla\left(\frac{h'(\rho)\nabla\sqrt{\rho}}{\sqrt{\rho}}\right)\right|^2+\frac{1}{2}\iint\frac{(h'(\rho))^3|\nabla\sqrt{\rho}|^4}{\rho}.
\end{aligned}
\end{equation}
Then, by \eqref{eq:sd6} we get 
\begin{equation}\label{eq:sd9}
\iint h'(\rho)|\nabla(h'(\rho)\nabla\sqrt{\rho})|^2+\iint\frac{(h'(\rho))^3|\nabla\sqrt{\rho}|^4}{\rho}\leq C\iint h(\rho)|\nabla^2\phi(\rho)|^2.
\end{equation}
Now we prove the second part of the Lemma. Assuming that $h'(\rho)>c$ it is straightforward to prove that 
\begin{equation*}
\iint|\nabla\rho^{\frac{1}{4}}|^4\leq C\iint h(\rho)|\nabla^2\phi(\rho)|^2.
\end{equation*}
By using the chian rule we have

\begin{equation}\label{eq:sd10}
\begin{aligned}
\iint h'(\rho)|\nabla(h'(\rho)\nabla\sqrt{\rho})|^2&=\iint h'(\rho)|h'(\rho)\nabla^{2}\sqrt{\rho}+2h''(\rho)\sqrt{\rho}\nabla\sqrt{\rho}\otimes\nabla\sqrt{\rho}|^2\\
                                                                        &=\iint h'(\rho)|h'(\rho)\nabla^{2}\sqrt{\rho}|^2\\
                                                                        &+4\iint h'(\rho)\frac{h'(\rho)\nabla^{2}\sqrt{\rho}h''(\rho)\rho\nabla\sqrt{\rho}\otimes\nabla\sqrt{\rho}}{\sqrt{\rho}}\\
                                                                        &+4\iint h'(\rho)|h''(\rho)\sqrt{\rho}\nabla\sqrt{\rho}\otimes\nabla\sqrt{\rho}|^2.
\end{aligned}
\end{equation}
By using the fact that $|h''(\rho)\rho|\leq C h'(\rho)$ we get 
\begin{equation}\label{eq:sd10}
\begin{aligned}
\iint h'(\rho)|h'(\rho)\nabla^{2}\sqrt{\rho}|^2&\leq C\iint h'(\rho)|\nabla(h'(\rho)\nabla\sqrt{\rho})|^2\\
                                                                  &+4\left|\iint\iint h'(\rho)\frac{h'(\rho)\nabla^{2}\sqrt{\rho}h''(\rho)\rho\nabla\sqrt{\rho}\otimes\nabla\sqrt{\rho}}{\sqrt{\rho}}\right|\\
                                                                  &\leq C\iint h'(\rho)|\nabla(h'(\rho)\nabla\sqrt{\rho})|^2+C\iint\frac{(h'(\rho))^3|\nabla\sqrt{\rho}|^4}{\rho}\\
                                                                  &+\frac{1}{2}\iint  h'(\rho)|h'(\rho)\nabla^{2}\sqrt{\rho}|^2.
\end{aligned}
\end{equation}
Then, by using \eqref{eq:sd9} we get 
\begin{equation}\label{eq:sd11}
\iint h'(\rho)|h'(\rho)\nabla^{2}\sqrt{\rho}|^2\leq C\iint h(\rho)|\nabla^2\phi(\rho)|^2.
\end{equation}
Then, since $h'(\rho)>c$ we have that 
\begin{equation}\label{eq:sd12}
\iint |\nabla^{2}\sqrt{\rho}|^2\leq C\iint h(\rho)|\nabla^2\phi(\rho)|^2.
\end{equation}
Summing up \eqref{eq:sd11} and \eqref{eq:sd12} we get \eqref{eq:lem51b}.
\end{proof}
\bigskip
\subsection{Preliminary Lemma}
In the following lemma we prove the main convergences needed in the proof of the main Theorems. 
\begin{lemma}\label{lem:sr}
Let $\{(\rn, \un)\}_\e$ be a sequence of solutions of \eqref{eq:aqns}. Then up to subsequences there exists a function $\rrho$ such that
\begin{align}
&\rrn\rightarrow\rrho\textrm{ strongly in }L^{2}(0,T; H^{1}(\T)).\label{eq:strong1}
\end{align}
\end{lemma}
\begin{proof}
Let us consider the first equation in \eqref{eq:aqns}. Since, by Proposition \ref{prop:DG} $\rho_\eps>0$ we have that
\begin{equation}\label{eq:51}
\partial_t\rrn=-\frac{\rrn}{2}\dive\un-\dive(\rrn\un)+\nabla\un\rrn,
\end{equation}
and, by the uniform bounds we have in \eqref{eq:uf2} and \eqref{eq:uf4}, we have that
\begin{equation*}
\{\partial_t\rrn\}_\e\textrm{ is uniformly bounded in }L^{2}(0,T;H^{-1}(\T)).
\end{equation*} 
Then, since $\{\rrn\}_\e$ is uniformly bounded in $L^{2}(0,T;H^{2}(\T))$ by using Aubin-Lions Lemma we get \eqref{eq:strong1}. 
\end{proof}
\begin{lemma}\label{lem:srr}
Let $\{(\rn, \un)\}_\e$ be a sequence of solutions of \eqref{eq:aqns}. Then
\begin{align}
&\he'(\re)\rrn\rightarrow\rrho\textrm{ strongly in }L^{2}((0,T)\times\T)),\label{eq:strong1bis}\\
&\he'(\rn)\nabla\rrn\rightarrow\nabla\rrho\textrm{ strongly in }L^{2}((0,T)\times\T),\label{eq:strongmain}\\
&\he''(\rn)\rn\nabla\rrn\rightarrow 0\textrm{ strongly in }L^{2}((0,T)\times\T).\label{eq:strongmainbis}
\end{align}
\end{lemma}
\begin{proof}
Let us start by proving \eqref{eq:strong1bis}. By using \eqref{eq:hrho} we get 
\begin{equation*}
\begin{aligned}
\iint|\he'(\re)\rrn-\rre|^2&\leq \iint|\rre-\rrho|^2\\
                                 &+\e^2\iint\re^{\frac{3}{4}}\\
                                 &+\e^2\iint\re^{2\g-1}.
\end{aligned}
\end{equation*}
The first term goes to zero because of \eqref{eq:strong1}. The second term, simply by using H\"older inequality and the uniform bound \eqref{eq:uf1}. Finally, for the last term we have that when $d=2$ there for any $\g>1$ fixed there exists $\delta=\delta(\g)$ small enough such that $2\g-1<(2-\delta)\g$ and them the integral is bounded because $\re^{\g}\in L^{r}((0,T)\times\T)$ for any $r<2$. When $d=3$ since $\g\in(1,3)$ 
it holds that $2\g-1<\frac{5}{3}\g$. Then, the third term goes to zero by using H\"older inequality and 
\eqref{eq:uf5}.\par
To prove \eqref{eq:strongmain} we have
\begin{equation*}
\begin{aligned}
\iint|\he'(\re)\nabla\rre-\nabla\rrho|^2&\leq \iint|\nabla\rre-\nabla\rrho|^2\\
                                                         &+C\e^2\iint\re^{-\frac{1}{4}}|\nabla\rre|^2\\
                                                         &+C\e^2\iint\re^{2\g-2}|\nabla\rre|^2.
\end{aligned}
\end{equation*}
Then, the first term goes to $0$ because of \eqref{eq:strong1}. Concerning the second term we have 
\begin{equation*}
\begin{aligned}
\e^2\iint\re^{-\frac{1}{4}}|\nabla\rre|^2&\leq \e^{2}\iint\re^{\frac{1}{4}}|\nabla\re^{\frac{1}{4}}|^2.
\end{aligned}
\end{equation*}
Then, by H\"older inequality 
\begin{equation*}
\e^2\iint\re^{-\frac{1}{4}}|\nabla\rre|^2\leq C\e^{2}\left(\iint\rre\right)^{\frac{1}{2}}\left(\iint|\nabla\re^{\frac{1}{4}}|^4\right)^{\frac{1}{2}}\leq C\e^{2}.
\end{equation*}
Now, we treat the last term. From \eqref{eq:uf3} we have that
\begin{equation*}
\iint|\nabla\re|^2\he'(\re)\re^{\g-2}\leq C.
\end{equation*}
Then, by using \eqref{eq:hrho} we get 
\begin{equation*}
\e\iint|\nabla\re|^2\re^{2\g-3}=4\e\iint|\nabla\rre|^2\re^{2\g-2}\leq C,
\end{equation*}
which implies the convergence of the last term. Finally concerning \eqref{eq:strongmainbis}, by using again \eqref{eq:hrho}, we have 
\begin{equation*}
\begin{aligned}
\iint|\he''(\re)\re\nabla\rre-\nabla\rrho|^2&\leq C\e^2\iint\re^{-\frac{1}{4}}|\nabla\rre|^2\\
                                                         &+C\e^2\iint\re^{2\g-2}|\nabla\rre|^2.
\end{aligned}
\end{equation*}
Then, it goes to zero arguing as above.
\end{proof}
\begin{lemma}\label{lem:shg}
Let $\{(\rn, \un)\}_\e$ be a sequence of solutions of \eqref{eq:aqns} then
\begin{align}
&\he(\re)-\re\rightarrow 0\textrm{ in }L^{1}((0,T)\times\T),\label{eq:convh}\\
&\gie(\re)\rightarrow 0\textrm{ in }\textrm{ in }L^{1}((0,T)\times\T).\label{eq:convg}
\end{align}
\end{lemma}
\begin{proof}
By \eqref{eq:hrho} and \eqref{eq:grho} we have that 
\begin{equation*}
\begin{aligned}
\iint|\he(\re)-\re|&\leq C\e\iint\re^{\frac{7}{8}}+C\e\iint\re^{\g},\\
\iint|\gie(\re)|&\leq  C\e\iint\re^{\frac{7}{8}}+C\e\iint\re^{\g}.
\end{aligned}
\end{equation*}
Then, we conclude by using H\"older inequality and \eqref{eq:uf1}. 
\begin{lemma}\label{lem:sp}
Let $\{(\rn, \un)\}_\e$ be a sequence of solutions of \eqref{eq:aqns} then
\begin{equation}
\begin{aligned}
&\re^{\g}\rightarrow\rho^{\g}\textrm{ in }L^{1}((0,T)\times\T),\\
&p_{\e}(\re)\rightarrow 0\textrm{ in }L^{1}((0,T)\times\T),\\
&\tilde{p}_{\e}(\re)\rightarrow 0\textrm{ in }L^{1}((0,T)\times\T).
\end{aligned}
\end{equation}
\end{lemma}
\begin{proof}
The convergence of $\re^{\g}$ follows from \eqref{eq:strong1} and the bound \eqref{eq:uf5bis}.
By the definition of $p_{\e}$ we have that there exists a generic constant $C$ independent on $\e$ such that 
\begin{equation}
\iint|p_{\e}(\re)|\leq C\sum_{i=1}^6\iint|p^{i}_{\e}(\re)|.
\end{equation}
Let us recall from \eqref{eq:uf1}
\begin{equation}\label{eq:ppre1}
\sup_{t}\e^5\lambda(\e)\left(\int\re^{\frac{1}{\e^2}+\g-1}+\re^{-\frac{1}{\e^2}-\frac{1}{8}}\right)\leq C.
\end{equation}
We start by estimating $p_{\e}^{1}(\re)$, by \eqref{eq:defp} and H\"older inequality we have
\begin{equation}
\begin{aligned}
\iint|p^{1}_{\e}(\re)|\leq C\e^{2}\lambda(\e)\iint\re^{\frac{1}{\e^2}}&
\leq C\e^{2}\lambda(\e)\left(\iint\re^{\frac{1}{\e^2}+\g-1}\right)^{\frac{1}{1+\e^2(\g-1)}}.
\end{aligned}
\end{equation}
Then, by using that $\lambda(\e)=e^{-1/\e^4}$ we get
\begin{equation*}
\int p_{\e}^{1}(\re)\leq 
C\frac{\e^2}{\e^{5/(1+\e^2(\g-1))}}e^{-\left(\frac{(\g-1)}{\e^2(1+\e^2(\g-1))}\right)}\left(\e^5\lambda(\e)\sup_t\int\re^{\frac{1}{\e^2}+\g-1}\right)^{\frac{1}{1+\e^2(\g-1)}}
\end{equation*}
and then by using \eqref{eq:ppre1} we get the convergence to $0$ of $p_{\e}^{1}(\re)$. 
The term $p_{\e}^{2}(\re)$ is treated at the same way. 
Now, we deal with convergence of the term $p^{3}_{\e}(\re)$. 
First of all, we have that 
\begin{equation*}
\iint|p^{3}_{\e}(\re)|\leq C \e^{3}\lambda(\e)\iint \re^{\frac{1}{\e^2}+\g-1}.
\end{equation*}
Then, we recall from \eqref{eq:uf3} the following uniform bound
\begin{equation}\label{eq:pre1}
\iint|\nabla\re|^2\he'(\re)f_{\e}''(\re)\leq C 
\end{equation}
which contains the following uniform bound
\begin{equation*}
\e^{10}\lambda(\e)\iint|\nabla\re|^2\re^{\frac{1}{\e^2}+2\g-4}\leq C
\end{equation*}
which means that 
\begin{equation*}
\e^{14}\lambda(\e)\iint\left|\nabla\left(\re^{\frac{1}{2\e^2}+\g-1}\right)\right|^2\leq C.
\end{equation*}
Then, by Sobolev embedding
\begin{equation*}
\e^{14}\lambda(\e)\int\left(\int\re^{\frac{3}{\e^2}+6\g-6}\right)^{\frac{1}{3}}\,dt\leq C.
\end{equation*}
Now, by H\"older inequality we get 
\begin{equation*}
\begin{aligned}
\iint|p_{\e}^{3}(\re)|&\leq C\e^{3}\lambda(\e)\left(\int\left(\int\re^{\frac{3}{\e^2}+6\g-6}\right)^{\frac{1}{3}}\,dt\right)^{\frac{1+\e^2(\g-1)}{1+2\e^2(\g-1)}}\\
&\leq C\frac{\e^3}{\e^{\frac{14(1+\e^2(\g-1))}{1+2\e^2(\g-1)}}}e^{-\frac{(\g-1)}{\e^2(1+2\e^2(\g-1))}}
\left(\e^{14}\lambda(\e)\int\left(\int \re^{\frac{3}{\e^2}+6\g-6}\right)^{\frac{1}{3}}\,dt\right)^{\frac{1+\e^2(\g-1)}{1+2\e^2(\g-1)}}.
\end{aligned}
\end{equation*}
Then, we have that $p^{3}_{\e}(\re)$ vanishes as $\e$ goes to $0$. 
Let us consider the term $p^{4}_{\e}(\re)$. We have 
\begin{equation*}
\begin{aligned}
\iint|p^{4}_{\e}(\re)|\leq C\e^{3}\lambda(\e)\iint\re^{-\frac{1}{\e^2}}&
\leq C\e^{3}\lambda(\e)\left(\iint\re^{-\frac{1}{\e^2}-\frac{1}{8}}\right)^{\frac{8}{8+\e^2}}\\
&\leq C \frac{\e^3}{\e^\frac{40}{8+\e^2}}e^{-\frac{1}{\e^2(8+\e^2)}}\left(\e^5\lambda(\e)\sup_t\int\re^{-\frac{1}{\e^2}-\frac{1}{8}}\right)^{\frac{8}{8+\e^2}}.
\end{aligned}
\end{equation*}
Then we get that $p_{\e}^{4}(\re)$ goes to $0$.
Now, we consider the term $p_{\e}^{5}(\re)$. By \eqref{eq:defp} and H\"older inequality
\begin{equation}
\iint|p_{\e}^{5}(\re)|\leq C \e^{3}\lambda(\e)\iint\re^{-\frac{1}{\e^2}-\frac{1}{8}}.
\end{equation}
In the bound \eqref{eq:pre1} is contained the following bound
\begin{equation*}
\e^{10}\lambda(\e)\iint\re^{-\frac{1}{8}-1}|\nabla\re|^2\re^{-\frac{1}{\e^2}-\frac{1}{8}-1}
\leq C
\end{equation*}
which means that 
\begin{equation*}
\e^{14}\lambda(\e)\iint\left|\nabla\left(\re^{-\frac{1}{2\e^2}-\frac{1}{8}}\right)\right|^2\leq C
\end{equation*}
which by Sobolev embedding implies 
\begin{equation}
 \e^{14}\lambda(\e)\int\left(\int\re^{-\frac{3}{\e^2}-\frac{6}{8}}\,dx\right)^{\frac{1}{3}}\,dt\leq C.
\end{equation}
Then,
\begin{equation} 
\iint|p_{\e}^{5}(\re)|\leq C \frac{\e^{3}}{\e^{\frac{14(8+\e^2)}{8+2\e^2}}}e^{-\frac{8+\e^2}{\e^2(8+2\e^2)}}\left(\e^{14}\lambda(\e)\int\left(\int\re^{-\frac{3}{\e^2}-\frac{6}{8}}\right)^{\frac{1}{3}}\,dt\right)^{\frac{8+\e^2}{8+2\e^2}}.
\end{equation}
and then $p^{5}_{\e}(\re)$ vanishes as $\e$ goes to $0$. 
Finally, the term $p_{\e}^{6}(\re)$ is treated as the term $p_{\e}^{4}(\re)$. The same proof of the convergence of the term $p_{\e}^{1}(\re)$ and $p_{\e}^{4}(\re)$ show the convergence of the damping coefficient $\tilde{p}_{\e}(\re)$. 
\end{proof}
\begin{lemma}\label{lem:sm}
Let $\{(\rn, \un)\}_\e$ be a sequence of solutions of \eqref{eq:aqns} then up to subsequences there exists a vector $m_1$ such that 
\begin{equation}\label{eq:strong3}
\rn\un\rightarrow m_1\textrm{ strongly in }L^{2}(0,T;L^{p}(\T))\textrm{ with }p\in[1,3/2).
\end{equation}
\end{lemma}
\begin{proof}
To prove \eqref{eq:strong3} we first notice that from the bounds \eqref{eq:uf1} and \eqref{eq:uf3}
\begin{equation}\label{eq:51bis}
\{\nabla(\rn\un)\}_\e\textrm{ is uniformly bounded in }L^{2}(0,T;L^{1}(\T)).
\end{equation}
Then, we need to estimate the time derivative of $\re\ue$. Precisely, we are going to prove that 
\begin{equation}
\int\|\partial_t(\re\ue)\|_{W^{-1,1}}\leq C
\end{equation}
By using the first equation in \eqref{eq:aqns} we get 
\begin{equation}
\begin{aligned}
\partial_{t}(\re\ue)&=-\dive(\re \ue\otimes \ue)-\nabla\re^{\g}-\nabla p_{\e}(\re)-\tilde{p}_{\e}(\re)\ue\\
                            &+2\nu\dive(\he(\re)D\ue)+2\nu\nabla(\gie(\re)\dive \ue)\\
                            &+\kappa^2\dive(\he(\re)\nabla^{2}(\phie(\re))+\nabla(\gie(\re)\Delta\phie(\re))\\
                            &=\sum_{i=1}^{8}I^{\e}_{i}
\end{aligned}
\end{equation}
First of all we notice that 
\begin{equation}\label{eq:mom1}
\begin{aligned}
&|\gie(\re)|\leq C \he(\re)\\
&\he(\re)\textrm{ is uniformly bounded in }L^{1}_{t,x}.
\end{aligned}
\end{equation} 
Then, we estimates each term. From \eqref{eq:uf1} we have that 
\begin{equation}
\{I^{\e}_1\}_\e\textrm{ is uniformly bounded in }L^{1}(0,T;W^{-1,1}(\T)).
\end{equation}                        
By using Lemma \ref{lem:sp} we get that for $i=2,3,4$ 
\begin{equation}
\{I^{\e}_i\}_\e\textrm{ is uniformly bounded in }L^{1}(0,T;W^{-1,1}(\T)).
\end{equation}   
Regarding the stress tensor by using \eqref{eq:mom1} and \eqref{eq:uf4} we have for $i=5, 6$
\begin{equation}
\{I^{\e}_i\}_\e\textrm{ is uniformly bounded in }L^{1}(0,T;W^{-1,1}(\T)).
\end{equation}    
Then, by using the \eqref{eq:uf3} and \eqref{eq:mom1} we have also that for $i=7,8$
\begin{equation}
\{I^{\e}_i\}_\e\textrm{ is uniformly bounded in }L^{1}(0,T;W^{-1,1}(\T)).
\end{equation}
In particular, this implies that $\{\partial_{t}(\re\ue)\}_{\e}$ is uniformly bounded in $L^{1}(0,T;W^{-1,1}(\T))$, then by standard Aubin-Lions Lemma we get \eqref{eq:strong3}.
\end{proof}
\begin{lemma}\label{lem:smain}
Let $(\rn,\un)$ be a sequence of solutions of \eqref{eq:qns} and let $\wn=\un+c\nabla\phie(\re)$. Then, up to subsequences we have that 
\begin{equation}\label{eq:main2}
\rrn\un\rightarrow\sqrt{\rho}u\textrm{ strongly in }L^{2}((0,T)\times\T),
\end{equation}
where $u$ is defined $m/\rho$ on $\{\rho>0\}$ and $0$ on $\{\rho=0\}$. 
\end{lemma}
\begin{proof}
Let us consider the Mallet-Vasseur type estimate \eqref{eq:mvinequality} in Proposition \ref{prop:mvinequality}. By using \eqref{eq:uf1}-\eqref{eq:uf5bis} and by taking $\delta>0$ sufficiently small in \eqref{eq:mvinequality} we may infer
\begin{equation}\label{eq:mvbis}
\sup_{t\in (0,T)}\int\re\left(1+\frac{|\we|^2}{2}\right)\log\left(1+\frac{|\we|^2}{2}\right)\leq C.
\end{equation}
By Lemma \ref{lem:sr} and Lemma \ref{lem:srr} we can extract a further subsequence such that 
\begin{equation}\label{eq:convmom}
\begin{aligned}
&\rrn\rightarrow\rrho\textrm{ a.e. in }(0,T)\times\T,\\
&\nabla\rrn\rightarrow\nabla\rrho\textrm{ a.e. in }(0,T)\times\T,\\
&\he'(\re)\nabla\rrn\rightarrow\nabla\rrho\textrm{ a.e. in }(0,T)\times\T,\\
&m_{1,\e}=\rn\un\rightarrow m_1\textrm{ a.e. in }(0,T)\times\T.
\end{aligned}
\end{equation}
Then, it follows that 
\begin{equation}\label{eq:convmomw}
m_{2,\e}:=m_{1,\e}+2\mu\rrn \he'(\re)\nabla\rrn\rightarrow m_1+2\mu\rrho\nabla\rrho=:m_2,
\end{equation}
a.e. in $(0,T)\times\T$.
Arguing as in \cite{MV} by using \eqref{eq:uf1}-\eqref{eq:uf4} and Fatou Lemma we have that 
\begin{equation}
 \iint\liminf_\e \frac{m_{1,\e}^2}{\rn}\leq \liminf_\e\iint \frac{m_{1,\e}^2}{\rn}< \infty.
\end{equation}
This implies that $m_{1}=0$ a.e. on $\{\rho=0\}$. Let us define the following limit velocity
\begin{equation*}
u=\left\{
\begin{array}{cc}
\displaystyle{\frac{m_1}{\rho}}  & \textrm{ on }\{\rho>0\},  \\
\\
0  & \textrm{ on }\{\rho=0\}.  \\
\end{array}
\right. 
\end{equation*}
In this way we have that $m_1=\rho u$ and $m_1/\sqrt{\rho}\in L^{\infty}(0,T;L^{2}(\T))$. 
Then from \eqref{eq:convmomw} we have that $m_2=m_1+2\mu\rrho\nabla\rrho$ and since $\nabla\rrho$ is finite almost everywhere we also have that $m_2=0$ on the set $\{\rho=0\}$. This in turn implies that after defining the following limit velocity 
\begin{equation*}
w=\left\{
\begin{array}{cc}
\displaystyle{\frac{m_1}{\rho}+2\mu\frac{\rrho\nabla\rrho}{\rho}}  & \textrm{ on }\{\rho\not=0\}  \\
\\
0  & \textrm{ on }\{\rho=0\},  \\
\end{array}
\right. 
\end{equation*}
we have that $m_2=\rho w$ and 
\begin{equation*}
\frac{m_2}{\rrho}=\rrho u+2\mu\nabla\rrho\in L^{\infty}(0,T;L^{2}(\T)).
\end{equation*}
Now we can prove \eqref{eq:main2}. First, by using \eqref{eq:convmom}, \eqref{eq:mvbis} and Fatou Lemma we get that 
\begin{equation}\label{eq:mvfinal}
\sup_t\int\rho|w|^2\log\left(1+\frac{|w|^2}{2}\right)\leq\sup_t\int\rn|\wn|^2\log\left(1+\frac{|\wn|^2}{2}\right)\leq C.
\end{equation}
Then, we note that for any fixed $M>0$
\begin{equation}\label{eq:convtro}
\rrn\wn\chi_{|\wn|\leq M}\rightarrow\rrho w\chi_{|w|\leq M}
\end{equation}
a.e. in $(0, T)\times\Omega$.
Indeed, in $\{\rho\not=0\}$ it holds
\begin{equation}
\rrn\wn=\frac{m_{1,\e}}{\rrn}+\he'(\re)\nabla\rrn\rightarrow\frac{m_1}{\rrho}+\nabla\rrho\textrm{   a.e.}
\end{equation} 
While, in $\{\rho=0\}$ we have 
\begin{equation}
|\rrn\wn\chi_{|\wn|<M}|\leq M\rrn\rightarrow 0\textrm{   a.e.}
\end{equation}
Then, 
\begin{equation*}
\begin{aligned}
\iint|\rrn\un-\rrho u|^2&\leq \iint|\rrn\wn-\rrho w|^2+4\mu^2\iint|\he'(\re)\nabla\rrn-\nabla\rrho|^2\\
                                 &\leq\iint|\rrn\wn\chi_{|\wn|<M}-\rrho w\chi_{|w|<M}|^2\\
                                 &+2\iint|\rrn\wn|^2\chi_{|\wn|>M}+2\iint|\rrho w|^{2}\chi_{|w|>M}\\
                                 &+4\mu^2\iint|\he'(\re)\nabla\rrn-\nabla\rrho|^2\\
                                 &\leq\iint|\rrn\wn\chi_{|\wn|<M}-\rrho w\chi_{|w|<M}|^2\\
                                 &+4\mu^2\iint|\he'(\re)\nabla\rrn-\nabla\rrho|^2\\
                                 &+\frac{C}{\log(1+M)}\iint\rn|\wn|^2\log\left(1+\frac{|\wn|^{2}}{2}\right)\\
                                 &+\frac{C}{\log(1+M)}\iint\rho|w|^2\log\left(1+\frac{|w|^2}{2}\right).
\end{aligned}                                 
\end{equation*}
By keeping $M$ fixed, we see that the first term on the right hand side goes to zero as $\eps\to0$ by the dominated convergence theorem. Furthermore also the second term goes to $0$ as $\eps\to0$ because of \eqref{eq:strong1}. Hence we have
\begin{equation*}
\lim_{\eps\to0}\iint|\sqrt{\rho_\eps}u_\eps-\sqrt{\rho}u|^2\leq\frac{C}{\log(1+M)},
\end{equation*}
for any $M>0$. By letting $M\to\infty$ we then get \eqref{eq:main2}.
\end{proof}
\bigskip
\subsection{Proof of Theorem \ref{teo:main1} and Theorem \ref{teo:main2}}
Let $(\rho^0, u^{0})$ be initial data for \eqref{eq:qns} satisfying \eqref{eq:hyidr} and \eqref{eq:hyidu} and let $(\re^0,\ue^0)$ be the sequence of initial data constructed in Section \ref{sec:app} satisfying \eqref{eq:hyaid}. For any  $\e<\e_f$ by using Theorem \ref{teo:main3} in the two dimensional case and Theorem \ref{teo:main4} in the three dimensional one, there exists a sequence of global smooth solutions $\{(\re,\ue)\}_{\e}$, $\re>0$, of \eqref{eq:aqns}-\eqref{eq:aid} and $(\rho,u)$, with $u$ defined zero on the set $\{\rho=0\}$, such that the convergences stated in Lemma \ref{lem:sr}-\ref{lem:smain} hold. We still denote with $(\re,\ue)$ the subsequence chosen in the convergence lemma.  Let us prove that $(\rho, u)$ is a finite energy weak solution of \eqref{eq:qns}-\eqref{eq:id}. Let us consider the first equation of \eqref{eq:aqns}. 
\begin{equation*}
\partial_t\re+\dive(\re\ue)=0.
\end{equation*}
The convergence to the weak formulation of \eqref{eq:qns} holds because of Lemma \ref{lem:sr} and Lemma \ref{lem:sm}. 
Next, let us consider the momentum equation 
\begin{equation*}
\begin{aligned}
&\partial_t(\re \ue)+\dive(\re \ue\otimes \ue)-2\nu\dive(\re D\ue)+\nabla\re^{\g}-\kappa^2\dive\K_{\e},\\
&=2\nu\dive((\he(\re)-\re)D\ue)+2\nu\nabla(\gie(\re)\dive \ue)\\
&-\nabla p_{\e}(\re)-\tilde{p}_{\e}(\re)\ue.
\end{aligned}
\end{equation*}
Then, 
\begin{equation*}
\iint|\he(\re)-\re||D\ue|\leq C \left(\iint |\he(\re)-\re| \right)^{\frac{1}{2}}\left(\iint|\he(\re)||D\ue|^2+|\re||D\ue|^2\right)^{\frac{1}{2}}
\end{equation*}
and this term converges to zero because of Lemma \ref{lem:shg}, \eqref{eq:uf1} and \eqref{eq:uf2}. Then, 
\begin{equation*}
\begin{aligned}
\iint|\gie(\re)||\dive \ue|&\leq C\left(\iint |\gie(\re)|\right)^{\frac{1}{2}}\left(\iint|\gie(\re)||\dive \ue|^2\right)^{\frac{1}{2}}\\
&\leq C \left(\iint |\gie(\re)|\right)^{\frac{1}{2}}\left(\iint|\he(\re)||\nabla \ue|^2\right)^{\frac{1}{2}}
\end{aligned}
\end{equation*}
and this term converges to zero because of Lemma \ref{lem:shg} and \eqref{eq:uf4}. Note that \eqref{eq:stin1} has been used.
Then, the pressure term $p_{\e}(\re)$ goes to $0$ because of Lemma \ref{lem:sp}. Concerning the damping term we have
\begin{equation*}
\iint|\tilde{p}_{\e}(\re)\ue|\leq \left(\iint |\tilde{p}_{\e}(\re)|\right)^{\frac{1}{2}}\left(\iint|\tilde{p}_{\e}(\re)||\ue|^2\right)^{\frac{1}{2}}\leq C\left(\iint |\tilde{p}_{\e}(\re)|\right)^{\frac{1}{2}}
\end{equation*}
which goes to zero thanks to Lemma \ref{lem:sp}. 
Now, we consider the terms in the left-hand side. The only convergence to prove is the convergence in the dispersive term. Indeed, since the strong convergence in $L^{2}_{t,x}$ of $\rre\ue$ holds the convergence of the other terms is straightforward, see \cite{AS} for more details.
Let us consider the following term where $\psi\in C^{\infty}([0,T)\times\T)$ 
\begin{equation*}
\begin{aligned}
\iint\dive\K_{\e}\cdot\psi&=\iint\nabla(\he'(\re)\Delta \he(\re))\psi-4\iint\dive(\he'(\re)\nabla\rre\otimes \he'(\re)\nabla\rre)\psi\\
                                      &=4\iint \he''(\re)\re \he'(\re)|\nabla\rre|^2\dive\psi\\
                                      &+2\iint \he'(\re)\rre \he'(\re)\nabla\rre\nabla\dive\psi\\
                                      &+4\iint \he'(\re)\nabla\rre\otimes \he'(\re)\nabla\rre\nabla\psi
\end{aligned}                                      
\end{equation*}
and by using Lemma \ref{lem:srr} it easy to conclude that 
\begin{equation*}
\begin{aligned}
\\
&4\iint \he''(\re)\re \he'(\re)|\nabla\rre|^2\dive\psi\rightarrow 0,\\
\\
&2\iint \he'(\re)\rre \he'(\re)\nabla\rre\nabla\dive\psi\rightarrow 2\iint\rrho\nabla\rrho\nabla\dive\psi,\\
\\
&4\iint \he'(\re)\nabla\rre\otimes \he'(\re)\nabla\rre\nabla\psi\rightarrow 4\iint\nabla\rrho\otimes\nabla\rrho:\nabla\psi.\\ 
\end{aligned}
\end{equation*}
\end{proof}
\section{Global Regularity for the approximating system}\label{sec:appex}

In this section we prove the global in time existence of smooth solutions for the approximating system \eqref{eq:aqns}. In the two-dimensional case the following theorem holds:
\begin{theorem}\label{teo:main3}
Let $\nu, \kappa>0$ such that $\kappa<\nu$ and $\g>1$. Then for  $\e<\e_{f}=\e_{f}(\nu, \kappa,\nu, \gamma)$ there exists a global smooth solution of \eqref{eq:aqns}-\eqref{eq:aid}. 
\end{theorem}
Concerning the three dimensional case, we have the following result 
\begin{theorem}\label{teo:main4}
Let $\nu, \kappa>0$ such that $\kappa^2<\nu^2<(9/8)\kappa^2$ and $\g\in(1,3)$. Then, for  $\e<\e_{f}=\e_{f}(\nu, \kappa,\nu, \gamma)$ there exists a global smooth solution of \eqref{eq:aqns}-\eqref{eq:aid}.
\end{theorem}
To prove Theorems \ref{teo:main3} and \ref{teo:main4} we use the standard theory of quasi-linear parabolic equations, see for example \cite{LSU}, in a similar spirit as it is also done in \cite{LX}, see Lemma 2.5 therein.
We only sketch the main ideas for the proofs, which otherwise would be quite long, and we focus only on proving the two main a priori estimates used to extend globally in time the local strong solution, namely Lemma \ref{prop:ldl4} and Proposition \ref{prop:DG}. It is possible to show the existence of a local strong solution to \eqref{eq:aqns}-\eqref{eq:aid} by a fixed point argument by considering the transformed system \eqref{eq:awqns}, expressed in terms of the mass density $\rho_\eps$ and the effective velocity $w_\eps=u_\eps+\mu\nabla\phi_\eps(\rho_\eps)$. For this purpose we write the system \eqref{eq:awqns} in the following way
\begin{equation}\label{eq:par}
\left\{\begin{aligned}
&\d_t\rho_\eps-\mu\Delta h_\eps(\rho_\eps)=F_1\\
&\d_tw_\eps-\nu\frac{h_\eps(\rho_\eps)}{\rho_\eps}\Delta w_\eps-\nu\frac{h_\eps(\rho_\eps)}{\rho_\eps}\nabla\diver w_\eps-(2\nu-\mu)h_\eps'(\rho_\eps)\nabla\diver w_\eps=F_2,
\end{aligned}\right.
\end{equation}
where the right hand side of the system is given by
\begin{equation*}
\left\{\begin{aligned}
F_1=&-\diver(\rho_\eps w_\eps)\\
F_2=&-w_\eps\cdot\nabla w_\eps-\frac1\rho_\eps\nabla p_\eps(\rho_\eps)-\frac1\rho_\eps\tilde p(\rho_\eps)w_\eps+2\mu\nabla h_\eps(\rho_\eps) Dw_\eps+2(\nu-\mu)\nabla h_\eps(\rho_\eps)\cdot Dw_\eps\\
&+(2\nu-\mu)\nabla g_\eps(\rho_\eps)\diver w_\eps.
\end{aligned}\right.
\end{equation*}
It is important to note that the right hand side of \eqref{eq:par} is uniformly parabolic since, by \eqref{eq:hrho} we have $h'_\eps(\rho_\eps), \frac{h_\eps(\rho_\eps)}{\rho_\eps}\geq1$.
Moreover, the terms $F_1$ and $F_2$ in the right-hand side are a semilinear perturbation of the left hand of term with lower order. Hence, a standard fixed point argument gives the existence of a time $T^*=T^*(\e, \rho^{0}_{\e}, w^0_{\e})$ such that, for any $T<T^{*}$ there exists a strong solution $(\rho_\eps, w_\eps)\in C([0, T);W^{2, p}(\mathbb T^3))$ to \eqref{eq:par} with $p>d$ such that  and $c_\eps\leq\rho_\eps(t, x)\leq C_\eps$ for some constants $c_\eps, C_\eps>0$, which depend on $T^*$.\par
The proof of the Theorems \ref{teo:main3} and \ref{teo:main4} then follows by a continuity argument, provided we prove that some a priori estimates hold uniformly with respect to the local existence time $T^*$. More precisely, we need to show \eqref{eq:novac}, namely that the density is uniformly bounded from above and below. We are going to prove \eqref{eq:novac} by exploiting the uniform parabolicity of the first equation in \eqref{eq:par} hence, in order to apply the classical regularity estimates, we will also need that
%
$\re\we, \we/\re\in L^{\infty}_t(L^{p}_x)$ with $p>d$. From the energy estimates we already know that both $\re$ and $1/\re$ are $L^{\infty}_t(L^{q}_x)$ for some very large $q$ depending on $\e$. Consequently it suffices to infer that $\re|\we|^{d+\delta}$ is in $L^{\infty}_{t}(L^{1}_{x})$; this is proved in the Lemma \ref{prop:ldl4} and it is exactly where we need the restriction on $\kappa$ and $\nu$ in three dimensions. This is due to the fact that in the second equation of \eqref{eq:awqns} there is the symmetric part of the gradient of $\we$. 

Once we obtain the uniform bounds \eqref{eq:novac} on the mass density, then an argument similar to Lemma 2.5 in \cite{LX} will yield the higher order estimates $(\rho_\eps, w_\eps)\in C([0, T);W^{2, p}(\mathbb T^3))$, where now $T>0$ is any finite time.
\begin{lemma}\label{prop:ldl4}
Let $(\re,\ue)$ be a smooth solution of \ref{eq:aqns}. Then, $(\re, \we)$ satisfies the following estimates:\\
In the two dimensional case, for any $\g>1$ and $\nu,\kappa>0$ there exists a small $\bar{\delta}=\bar{\delta}(\g, \nu, \kappa)$ and a constant $C$, possibly depending on $\e$,  such that for any $\delta<\bar{\delta}$ 
\begin{equation}\label{eq:ld}
\sup_{t}\int\re|\we|^{2+2\delta}\leq C. 
\end{equation}
In the three dimensional case, for any $\g\in(1,3)$ and $\nu, \kappa>0$ such that $\kappa^2<\nu^2<\frac{9}{8}\kappa^2$ there exists a constant $C$, possibly depending on $\e$, such that 
\begin{equation}\label{eq:l4}
\sup_{t}\int\re|\we|^{3+2\delta}\leq C.
\end{equation} 
\end{lemma}
\begin{proof}
For convenience of the reader we write again the integral equality of Lemma \ref{lem:renormalized}
\begin{equation*}
\begin{aligned}
&\frac{d}{dt}\int\re\be+\mu\int \he(\re)|A\we\cdot\we|^2\bedp+\mu\int \he(\re)|A \we|^2\bep\\
&+(2\nu-\mu)\int \he(\re)|D \we|^2\bep+(2\nu-\mu)\int \gie(\re)|\dive \we|^2\bep\\
&+\int\tilde{p}_{\e}(\re)|\we|^2\bep+(2\nu-\mu)\int \he(\re)|D \we\cdot \we|^2\bedp\\
&=-\int\nabla\re^{\g}\we\bep-2\nu\int \he(\re)(D\we\cdot\we)\cdot(A\we\cdot\we)\bedp\\
&-(2\nu-\mu)\int \gie(\re)\dive \we\we\cdot(D\we\cdot\we)\bedp.
\end{aligned}
\end{equation*}
Let $\beta(t)=t^{1+\delta}$ and with $\delta>0$. Then we have that $\beta'(t)=\frac{\beta''(t)t}{\delta}$. By integrating by parts the pressure term and using Young inequality we get 
\begin{equation*}
\begin{aligned}
&\frac{d}{dt}\int\rho\be+\mu\int \he(\re)|A\we\cdot\we|^2\bedp\\
&+(2\nu-\mu)\int \he(\re)|D\we\cdot\we|^2\bedp+(2\nu-\mu)\int \he(\re)|D \we|^2\bep\\
&+(2\nu-\mu)\int \gie(\re)|\dive \we|^2\bep+\mu\int \he(\re)|A \we|^2\bep\\
&+\int\tilde{p}_{\e}(\re)|\we|^2\bep\leq\frac{1}{2\delta}\int\re^{\g}\dive\we|\we|^2\bedp\\
&+\int\re^{\g}\we(D\we\cdot\we)\bedp+\nu\int \he(\re)|D\we\cdot\we|^2\bedp\\
&+\nu\int \he(\re)|A\we\cdot\we|^2\bedp+\frac{(2\nu-\mu)(\g-1)}{2}\e\int\re^{\g}|\dive\we|^2|\we|^2\bedp\\
&+\frac{(2\nu-\mu)(\g-1)}{2}\e\int\re^{\g}|D\we\cdot\we|^2\bedp+\frac{3(2\nu-\mu)}{16}\e\int\re^{\frac{7}{8}}|D\we|^2|\we|^2\bedp\\
&+\frac{2\nu-\mu}{16}\e\int\re^{\frac{7}{8}}|D\we\cdot\we|^2\bedp.
\end{aligned}
\end{equation*}
Then, by writing everything in term of $\bedp$, using \eqref{eq:stin1} and the fact that $|A\we\cdot\we|^2\leq |A\we|^2|\we|^2$ we have
\bigskip
\begin{equation*}
\begin{aligned}
&\frac{d}{dt}\int\re\be+\left(\mu+\frac{\mu}{2\delta}\right)\int \he(\re)|A\we\cdot\we|^2\bedp\\
&+(2\nu-\mu)\int \he(\re)|D\we\cdot\we|^2\bedp+\frac{2\nu-\mu}{2\delta}\e\int\re|D\we|^2|\we|^2\bedp\\
&+\left(\frac{2\nu-\mu}{2\delta}\right)\e\int\re^{\g}|D\we|^2|\we|^2\bedp+\frac{5(2\nu-\mu)}{16\delta}\e\int\re^{\frac{7}{8}}|D\we|^2|\we|^2\bedp\\
&+\frac{(2\nu-\mu)(\g-1)}{2\delta}\e\int\re^{\g}|\dive\we|^2|\we|^2\bedp\\
&\leq \frac{C(\tau)}{2\delta}\int\re^{2\g-1}|\we|^2\bedp+\frac{3\tau}{2\delta}\int\re|D\we|^2|\we|^2\bedp\\
&+C(\alpha)\int\re^{2\g-1}|\we|^2\bedp+\alpha\int\re|D\we\cdot\we|^2\bedp\\
&+\nu\int \he(\re)|A\we\cdot\we|^2\bedp+\left(\nu+\frac{(2\nu-\mu)}{16}\right)\e\int\re^{\frac{7}{8}}|D\we\cdot\we|^2\bedp\\
&+\left(\nu+\frac{(2\nu-\mu)(\g-1)}{2}\right)\e\int\re^{\g}|D\we\cdot\we|^2\bedp+\nu\int \re|D\we\cdot\we|^2\bedp\\
&+\frac{(2\nu-\mu)(\g-1)}{2}\e\int\re^{\g}|\dive\we|^2|\we|^2\bedp+\frac{3(2\nu-\mu)}{16}\e\int\re^{\frac{7}{8}}|D\we|^2|\we|^2\bedp.
\end{aligned}
\end{equation*}
Finally, by absorbing terms from the right hand-side to the left hand-side we get 
\begin{equation}\label{eq:l4ld1}
\begin{aligned}
&\frac{d}{dt}\int\re\be+\left(\mu-\nu+\frac{\mu}{2\delta}\right)\int \he(\re)|A\we\cdot\we|^2\bedp\\
&+(\nu-\mu)\int\re|D\we\cdot\we|^2\bedp\\
&+\left(\nu-\mu-\frac{(2\nu-\mu)(\g-1)}{2}\right)\e\int\re^{\g}|D\we\cdot\we|^2\bedp\\
&+\left(\frac{7\nu}{8}-\frac{15\mu}{16}\right)\e\int\re^{\frac{7}{8}}|D\we\cdot\we|^2\bedp\\
&+\frac{2\nu-\mu}{2\delta}\int\re|D\we|^2|\we|^2\bedp\\
&+\left(\frac{2\nu-\mu}{2\delta}\right)\e\int\re^{\g}|D\we|^2|\we|^2\bedp\\
&+\left(\frac{5(2\nu-\mu)}{16\delta}-\frac{3(2\nu-\mu)}{16}\right)\e\int\re^{\frac{7}{8}}|D\we|^2|\we|^2\bedp\\
&+\frac{(2\nu-\mu)(\g-1)}{2}\left(\frac{1}{\delta}-1\right)\e\int\re^{\g}|\dive\we|^2|\we|^2\bedp\\
&\leq \frac{C(\tau)}{2\delta}\int\re^{2\g-1}|\we|^2\bedp\\
&+\frac{3\tau}{2\delta}\int\re|D\we|^2|\we|^2\bedp\\
&+C(\alpha)\int\re^{2\g-1}|\we|^2\bedp\\
&+\alpha\int\re|D\we\cdot\we|^2\bedp\\
\end{aligned}
\end{equation}
By considering $\delta\in(0,1]$ and choosing $\tau=(2\nu-\mu)/6$ and $\alpha=\nu-\mu$, after using Young inequality we get 
\begin{equation}\label{eq:delta}
\begin{aligned}
\frac{d}{dt}\int\re|\we|^{2+2\delta}&+\left(\mu-\nu+\frac{\mu}{2\delta}\right)\int \he(\re)|A\we\cdot\we|^2\bedp\\
&+\left(\nu-\mu+\frac{2\nu-\mu}{2}\left(\frac{1}{\delta}-(\g-1)\right)\right)\e\int\re^{\g}|D\we\cdot\we|^2\bedp\\
&+\left(\frac{18\nu-17\mu}{16}\right)\e\int\re^{\frac{7}{8}}|D\we\cdot\we|^2\bedp\\
&\leq C(\nu,\kappa,\delta)\left(\int\re^{\tilde{\g}}+\int\re|\we|^{2+2\delta}\right)
\end{aligned}
\end{equation}
with 
\begin{equation*}
\tilde{\g}:=\left(2\g-1-\frac{\delta}{1+\delta}\right)(1+\delta)>0\textrm{ for any }\delta\in(0,1).
\end{equation*}
Now we consider the two dimensional case. Given $\nu,\kappa>0$ with $\nu>\kappa$ and $\g>1$ it is easy to find $\delta$ small enough 
\begin{equation*}
\begin{aligned}
&\mu-\nu+\frac{\mu}{2\delta}>0\quad&\frac{1}{\delta}-(\g-1)>0.
\end{aligned}
\end{equation*}
By Proposition \ref{prop:energy} we have that 
\begin{equation*}
\sup_{t}\int\re^{\frac{1}{\e^2}}\leq C(\e).
\end{equation*} 
Then, by choosing if needed $\e_f$ small enough such that for any $\e<\e_f$ it holds $\tilde{\g}<1/\e^2$ we arrive at
\begin{equation*}
\frac{d}{dt}\int\re|\we|^{2+2\delta}\leq C(\e,\kappa,\mu,\nu)+\int\re|\we|^{2+2\delta}
\end{equation*}
and we get \eqref{eq:ld} by Gronwall Lemma.
Then, we consider the three dimensional case. Since it seems not possible to avoid a restriction on $\nu, \kappa$ we do not aim to optimality, which can be obtained by optimizing the the Young inequalities and minimizing in $\delta$. Going back to \eqref{eq:delta} we argue as follows. First we note that for any $\delta\in(0,1)$ it holds 
\begin{equation}\label{eq:d1}
\mu-\nu+\frac{\mu}{2\delta}>\frac{3\mu}{2}-\nu
\end{equation}
and the left-hand side of \eqref{eq:d1} is positive if the following restriction on $\nu$ and $\kappa$ holds:
\begin{equation}\label{eq:restriction}
\kappa^2<\nu^2<\frac{9}{8}\kappa^2
\end{equation}
Then, by using that $\g\in(1,3)$  we can choose $\delta=\min\{1/(\g-1),1\}$. Notice that $\delta\in(1/2,1)$ and 
\begin{equation*}
\nu-\mu+\frac{2\nu-\mu}{2}\left(\frac{1}{\delta}-(\g-1)\right)>0
\end{equation*}
Then, since $\delta=\frac{1}{2}+\delta'$, with an abuse of notation avoiding the prime, we have
\begin{equation*}
\frac{d}{dt}\int\re|\we|^{3+2\delta}\leq C(\nu,\kappa,\delta)\left(\int\re^{\tilde{\g}}+\int\re|\we|^{3+2\delta}\right)
\end{equation*}
Then, by choosing if needed $\e_f$ small enough such that for any $\e<\e_f$ it holds $\tilde{\g}<1/\e^2$ we get \eqref{eq:l4} by using Gronwall Lemma. We stress that $\e_{f}$ depends only on $\g,\,\nu$ and $\kappa$.
\end{proof}
Now, we are in position to give the proof of Proposition \ref{prop:DG}.
\begin{proposition}\label{prop:DG}
Let $(\re,\ue)$ be a smooth solution of the system \eqref{eq:aqns}. then, there exists a constant $C>0$ possibly dependent on $\e$ such that 
\begin{equation}\label{eq:novac}
\frac{1}{C}\leq \re\leq C
\end{equation}
\end{proposition}

\begin{proof}
 First we want to prove that 
\begin{equation}\label{eq:dg1}
\re\we\textrm{ and }\frac{\we}{\re} \in L^{\infty}(0,T;L^{p}(\T)\textrm{ with }p>d
\end{equation} 
When $d=2$, by H\"older and Young inequality we get 
\begin{equation*}
\int|\re\we|^{2+\delta}\leq \int \re^{\tilde{\delta}}+\int\re|\we|^{2+2\delta}
\end{equation*}
with $\tilde{\delta}=\tilde{\delta}(\g,\nu,\kappa)>0$. Then by choosing if needed $\e$ small enough such that $\tilde{\delta}<1/\e^2$ we get the desired estimate. 
Then, 
\begin{equation*}
\left|\frac{|\we|}{\re}\right|^{2+\delta}=\left|\frac{1}{\re}\right|^{1+\frac{2+\delta}{2+2\delta}}|\re|^{\frac{2+\delta}{2+2\delta}}|\we|^{2+\delta}.
\end{equation*}
Then, by H\"older and Young inequality 
\begin{equation*}
\int\left|\frac{|\we|}{\re}\right|^{2+\delta}\leq \int\left|\frac{1}{\re}\right|^{\tilde{\delta}}+\int\re|\we|^{2+\delta},
\end{equation*}
with 
\begin{equation*}
\tilde{\delta}=\left(1+\frac{2+\delta}{2+2\delta}\right)\left(\frac{2+2\delta}{2+\delta}\right)^*.
\end{equation*}
Then, again by choosing if necessary $\e$ small enough such that $\tilde{\delta}\leq 1/\e^2$ we get the desired estimate. Now we are in position to prove \eqref{eq:novac}. The proof is standard and it is based on De Giorgi type estimate. We use the same approach as in \cite{LX}, Lemma 2.4. 
Let us start by proving that $\re$ is bounded. Let $m_\e=\re\we$. Then the first equation in \eqref{eq:awqns} is the following 
\begin{equation}
\label{eq:heatre}
\partial_t\re-\dive(\he'(\re)\nabla\re)=\dive m_{\e}.
\end{equation}
Let $k>\|\re^0\|_{\infty}$ and $A_{k}(t)=\{\re>k\}$. Then by using H\'older inequality and the fact that $\he'(\re)\ge 1$ we get 
\begin{equation}\label{eq:deg1}
\begin{aligned}
\frac{d}{dt}\int|(\re-k)_{+}|^2+\frac{1}{2}\int\he'(\re)|\nabla(\re-k)_{+}|^2&\leq\left(\int_{A_k(t)}|m_{\e}|^2\right)\\
&\left(\int|m_{\e}|^p\right)^{\frac{2}{p}}|A_{k}(t)|^{1-\frac{2}{p}}.
\end{aligned}
\end{equation}
By using \eqref{eq:dg1} and denoting $r_k=\sup_{t\in(0,T)}|A_k(t)|$ we get 
\begin{equation}\label{eq:deg2}
\frac{d}{dt}\int|(\re-k)_{+}|^2+\frac{1}{2}\int\he'(\re)|\nabla(\re-k)_{+}|^2\leq C\,r_{k}^{1-\frac{2}{p}}.
\end{equation}
Let $\sigma\in(0,T)$ such that 
\begin{equation*}
\int\int|(\re-k)_{+}(\sigma)|^2:=\sup_{t\in(0,T)}\int|(\re-k)_{+}(t)|^2.
\end{equation*}
Then, 
\begin{equation*}
\int|(\re-k)_{+}(\sigma)|^2+\int|\nabla(\re-k)_{+}(\sigma)|^2\leq r_{k}^{1-\frac{2}{p}},
\end{equation*}
where the fact that $\he'(\re)>1$ has been used. 
Let $l>k>\|\re^0\|_{\infty}$ and $q\geq 6$ to be chosen later. Then it holds that 
\begin{equation}\label{eq:main}
\begin{aligned}
|A_{l}(t)|(l-k)^2&\leq\|(\re-k)_{+}(t)\|_{2}^2\\
                                     &\leq\|(\re-k)_{+}(\sigma)\|_{2}^2\\
                                     &\leq\|(\re-k)_{+}(\sigma)\|_{q}^{2}|A_{k}(\sigma)|^{1-\frac{2}{q}}\\
                                     &\leq\|\nabla(\re-k)_{+}(\sigma)\|_{2}^2|A_{k}(\sigma)|^{1-\frac{2}{q}}\\
                                     &\leq Cr_{k}^{2-\frac{2}{p}-\frac{2}{q}}.
\end{aligned}
\end{equation}                                     
If we show that 
\begin{equation}\label{eq:dgmain}
r_{l}\leq(l-k)^2r_{k}^{1+\alpha}\textrm{ for some }\alpha>0
\end{equation}
by De Giorgi Lemma, see \cite[Lemma 4.1.1]{WYW}, we get that $\re$ is bounded. Then, in the three dimensional case by Sobolev embedding we are forced to take $q=6$ in \eqref{eq:main}. Then since $p>3$ we get \eqref{eq:dgmain} with 
\begin{equation*}
\alpha=\frac{2}{3}-\frac{2}{p} >0.            
\end{equation*}
Note that since $p$ is depending only on $\nu,\kappa$ and $\gamma$ then $\alpha$ has the dependence as well. In the two dimensional case by Sobolev embedding we can take any $q<\infty$. Then, given $p=p(\nu,\kappa,\gamma)>2$, it always possible to find $q$ big enough such that 
\begin{equation*}
\alpha=1-\frac{2}{p}-\frac{2}{q}>0.
\end{equation*}
Now we prove that $\re$ is bounded away from $0$. By using \eqref{eq:heatre} it follows that the equation for $\rei:=1/\re$ is the following. 
\begin{equation}
\partial_{t}\rei-\dive(\he'(\re)\nabla\rei)+2\frac{\he'(\re)|\nabla\re|^2}{\re^3}=-\dive\we\rei+\we\cdot\nabla\rei.
\end{equation}
Let $k>\|1/\re^0\|$, $m_{\e}:=\rei\we$ and $A_{k}(t):=\{\rei>k\}$. Then we get 
\begin{equation}\label{eq:dgir1}
\begin{aligned}
\frac{d}{dt}\int|(\rei-k)_{+}|^2&+\int\he'(\re)|\nabla(\rei-k)_{+}|^2+\int2\frac{\he'(\re)|\nabla\re|^2}{\re^3}\rei\\
                                             &=-\int\rei\dive\we(\rei-k)_{+}+\int\we\cdot\nabla\rei(\rei-k)_{+}\\
                                             &=2\int_{A_k(t)}\we\nabla\rei(\rei-k)_{+}+\int_{A_{k}(t)}\rei\we\cdot\nabla(\rei-k)_{+}\\
                                             &\leq2\int|\nabla(\rei-k)_{+}||\we||\rei|,
\end{aligned}
\end{equation}
where it has been used that $|(\rei-k)_{+}|\leq|\rei|$. By using H\"older inequality, Young inequality and the fact that $\he'(\re)\geq1$ we have that 
\begin{equation*}
\begin{aligned}
\frac{d}{dt}\int|(\rei-k)_{+}|^2+\frac{1}{2}\int\he'(\re)|\nabla(\rei-k)_{+}|^2&\leq\left(\int_{A_k(t)}|m_{\e}|^2\right)\\
&\left(\int|m_{\e}|^p\right)^{\frac{2}{p}}|A_{k}(t)|^{1-\frac{2}{p}},
\end{aligned}
\end{equation*}
with $p>d$ and the last term in the right-hand side of \eqref{eq:dgir1} has been dropped because it is positive. Then, by using \eqref{eq:dg1} and defining as before $r_{k}=\sup_{t\in(0,T)}|A_{k}(t)|$ we get 
\begin{equation*}
\frac{d}{dt}\int|(\rei-k)_{+}|^2+\frac{1}{2}\int\he'(\re)|\nabla(\rei-k)_{+}|^2\leq C\,r_{k}^{1-\frac{2}{p}}.
\end{equation*}
Arguing as in the proof of boundedness of $\re$ we can conclude that $\rei$ is bounded and then $\re$ is bounded away from $0$.
\end{proof}

\section*{Conflict of Interest}
The authors declare that they have no conflict of interest.

\end{document}